\documentclass[12pt, twoside,a4paper,reqno]{amsart}

\usepackage[a4paper,width=160mm,top=25mm,bottom=30mm]{geometry}
\usepackage{fancyhdr}
\usepackage{amsmath,amsthm,amssymb}
\usepackage{graphicx}
\usepackage{hyperref}
\usepackage{amsfonts}
\usepackage{tabularx}
\usepackage{enumerate}
\usepackage[normalem]{ulem}
\usepackage{color}

\usepackage{yhmath}
\usepackage{mathrsfs}
\usepackage[utf8]{inputenc}
\usepackage{dsfont}
\usepackage{yhmath}
\usepackage{mathrsfs}
\usepackage{mathtools}

\theoremstyle{plain}
\newtheorem{theorem}{Theorem}[section]
\newtheorem{corollary}[theorem]{Corollary}
\newtheorem{proposition}[theorem]{Proposition}

\newtheorem{lemma}[theorem]{Lemma}
\newtheorem*{claim}{Claim}

\theoremstyle{remark}
\newtheorem*{remark}{Remark}

\theoremstyle{definition}
\newtheorem{definition}{Definition}[section]

\mathtoolsset{showonlyrefs}
\newcommand{\eps}{\varepsilon}
\newcommand{\Z}{\mathbb{Z}}
\newcommand{\R}{\mathbb{R}}

\newcommand{\PP}{\mathbf{P}}
\renewcommand{\P}[1]{\PP \left [ #1 \right ]}
\newcommand{\Pp}[1]{\PP_p \left [ #1 \right ]}
\newcommand{\lr}[1][]{\overset{\:#1\:}\longleftrightarrow}
\usepackage{stmaryrd}
\newcommand{\nlr}[1][]{\overset{\:#1\:}{\longleftrightarrow}\!\!\!\!\!\!\!\!\!\arrownot
  \ \, \quad}

\newcommand{\cal}{\mathcal}
\newcommand{\ol}{\overline}
\renewcommand{\bar}{\overline}
\newcommand\blfootnote[1]{%
  \begingroup
  \renewcommand\thefootnote{}\footnote{#1}%
  \addtocounter{footnote}{-1}%
  \endgroup
}

\begin{document}
\title{Critical percolation and the Minimal spanning tree in slabs}
\author{Charles M. Newman}
\author{Vincent Tassion}
\author{Wei Wu}
\address[Charles Newman]{Courant Institute of Mathematical Sciences, New York University,
251 Mercer st, New York, NY 10012, USA \newline
\& NYU-ECNU Institute of Mathematical Sciences at NYU Shanghai, 3663 Zhongshan Road North, Shanghai 200062,
China.}
\address[Vincent Tassion]{D\'epartement de Math\'ematiques Universit\'e de Geneve, Geneve, Switzerland
}
\address[Wei Wu]{Courant Institute of Mathematical Sciences, New York University,
251 Mercer st, New York, NY 10012, USA \newline
\& NYU-ECNU Institute of Mathematical Sciences at NYU Shanghai, 3663 Zhongshan Road North, Shanghai 200062,
China.}
\date{\today }
\maketitle

\begin{abstract}
  The minimal spanning forest on $\Z^d$ is known to consist of a single tree for
  $d \leq 2$ and is conjectured to consist of infinitely many trees for large $d$. In this
  paper, we prove that there is a single tree for quasi-planar graphs such as
  $\Z^2\times {\{0,\ldots,k\}}^{d-2}$.
  Our method relies on generalizations of the ``Gluing Lemma'' of \cite{DST}. A
  related result is that critical Bernoulli percolation on a slab satisfies the
  box-crossing property. Its proof is based on a new Russo-Seymour-Welsh type theorem for
  quasi-planar graphs. Thus, at criticality, the probability of an open path
  from $0$ of diameter $n$ decays polynomially in $n$. This strengthens the
  result of \cite{DST}, where the absence of an infinite cluster at criticality
  was first established.
\end{abstract}

\section{Introduction}

 There are two standard models of
random spanning trees on finite graphs: the {\it uniform\/} spanning tree and
the {\it minimal\/} spanning tree. One can define these, by taking a limit, on
infinite graphs, such as $\Z^d$ with nearest-neighbor edges, but then the single
finite spanning tree may become a forest of many disjoint trees. Because the
uniform spanning tree is closely related to random walks and potential theory
(\cite{Wil}, see also \cite{BLPS}), it is known \cite{Pem} that the critical
dimension is exactly $d_{c}=4$, only above which is there more than a single
tree.

In the case of the minimal spanning tree, where random walks and potential theory
are replaced by invasion percolation and critical Bernoulli percolation,
very little is known rigorously. As we will discuss in more detail below,
it is known that there is a single tree in $\Z^2$ (\cite{CCN}, see also \cite{AM}) and there are conjectures that for large $d$
there are infinitely many trees. The main purpose of this paper is to make progress
toward a proof that at least for some low dimensions above $d=2$, there is a single
tree by showing that this is the case for the approximation of, say, $\Z^3$ by a thick
slab $\Z^2\times \{0,\ldots,k\}$. In the process, we also obtain a new result
for critical percolation on such slabs, where it was only recently proved that
there is no infinite cluster \cite{DST}; the new result is inverse power law
decay for the probability of large diameter finite clusters at criticality (see Corollary \ref{cor} in Section \ref{sec:RSW}).\blfootnote{\textbf{Bibliographic note:} After the current paper was finished, we learned that an alternate proof of the lower bound in the
  box-crossing property of Theorem~\ref{thm:BXP} (existence of open crossings in
  long rectangles with positive probability) was obtained
independently in the very recent paper \cite{sapozhnikov2015crossing}, using a different argument.}

To define the minimal spanning tree (MST) on
a finite connected graph $G=\left( V,E\right) $, assign random weights
$\left\{ \omega \left( e\right) :e\in E\right\} $, that are i.i.d.\@ uniform
$\left[ 0,1\right] $ random variables, to its edges. The MST
is the spanning tree that minimizes the total weight. Equivalently, it can be
obtained from $G$ by deleting every edge whose weight is maximal in some
cycle. When $G$ is an infinite graph, two natural infinite volume limits can be
taken, which lead to the notion of free and wired minimal spanning forests. See
\cite{Alex}, \cite{Hagg} and \cite{LPS} for basic properties of minimal spanning
forests on infinite graphs.

On $\mathbb{Z}^{d}$, it is known that the free and wired minimal spanning
forests coincide (see Proposition \ref{unique}). Therefore, in this framework,
we can talk about \emph{the} minimal spanning forest without ambiguity. Although
it arises as the weak limit of minimal spanning trees on finite graphs, the
minimal spanning forest may no longer be a single tree, and can even have
infinitely many components. A natural question is, for which $d$ is the minimal
spanning forest in $\mathbb{Z}^{d}$ almost surely a single tree? This question
is largely open except for $d=2$ (and trivially for $d=1$), where the minimal
spanning forest is known to be a single tree (\cite{CCN}, \cite{AM}); the argument there
crucially relies on the planarity, and does not apply to $d\geq 3$. A much more
modest question, whether the number of components in minimal spanning forest in
$\mathbb{Z}^{d}$ is either $1$ or $\infty $ almost surely, also remains open.
Besides its own interest, the number of connected components of minimal spanning
forests is also closely related to the ground state structure of
nearest-neighbor spin glasses and other disordered Ising models \cite{NS}. It is
believed that there is a finite upper critical dimension $d_{c}$, below which
the minimal spanning forest is a single tree a.s., and above which it has
infinitely many components a.s. Based on a combination of rigorous and heuristic
arguments, there have been interesting competing conjectures that $d_{c}=8$
\cite{NS} (see also \cite{NS94}) or $d_{c}=6$ \cite{Jack}. But it has not even been proved that there are
multiple trees for very large $d$.

Another natural random forest measure on infinite graphs is the uniform
spanning forest, defined as the weak limit of uniform spanning trees on
finite subgraphs. As mentioned earlier, the geometry of uniform spanning forests
is much better
understood, because of the connections to random walks and potential theory (%
\cite{Wil}, see also \cite{BLPS}). Its upper critical dimension is thus
closely related to the intersection probability of random walks. It was
shown in \cite{Pem} that here $d_{c}=4$.

Minimal spanning forests are closely related to critical Bernoulli percolation
and invasion percolation. Just as (wired) uniform spanning forests can be
constructed by piecing together loop erased random walks by Wilson's algorithm \cite{Wil},
(wired) minimal spanning forests can be constructed using invasion trees (see
Proposition \ref{invtree} below). On the two dimensional triangular lattice,
Garban, Pete and Schramm \cite{GPS} proved that the minimal spanning tree on
this graph has a scaling limit, based on fine knowledge of near-critical
Bernoulli percolation.

In this paper, we study the minimal spanning forests on a class of non-planar
infinite graphs, namely two dimensional slabs, whose vertex set is of the form
$\mathbb{Z}^{2}\times \left\{ 0,...,k\right\} ^{d-2}$, for $%
k\in \mathbb{N}$. Although $d>2$ can be arbitrary here, these graphs are all
quasi-planar. A main result of this paper is Theorem \ref{MSF}, which states
that on any two dimensional slab, the minimal spanning forest is a single tree
a.s. The argument also applies to other quasi-planar graphs, such as
$\mathbb{Z}^{2}$
with non-nearest neighbor edges up to a finite distance --- see the remark after Theorem \ref{MSF}.

An important ingredient in the proof is the box-crossing property for critical
Bernoulli percolation on slabs, stated as Theorem~\ref{thm:BXP}. Its proof is
based on a Russo-Seymour-Welsh type theorem, and extends to a larger class of
models --- e.g.,\@ Bernoulli percolation on quasi-planar graphs invariant under
a non trivial rotation, or short-range Bernoulli percolation on $\Z^2$ invariant
under $\pi/2$-rotations.

Because of the relation between the minimal spanning forest and critical
Bernoulli percolation, it is not surprising that we adapt tools from the
percolation literature. Indeed, a major open question in Bernoulli percolation
is to prove in $\mathbb{Z}^{d}$, $3\leq d \leq 6$, that there is no percolation at the
critical point. Although this question is still beyond reach, it was
recently proved in \cite{DST} that non-percolation at criticality is valid for two-dimensional
slabs. A key technical ingredient in that proof is a gluing lemma for open
paths (see Theorem \ref{GL} below for a more general version). In this paper we
use a related gluing lemma (Lemma %
\ref{glue}) that applies to invasion trees and minimal spanning trees.

This paper is organized as follows. In Section \ref{sec:back} we collect
definitions and basic properties for minimal spanning forests and invasion
percolation, describe their connections, state the main result (Theorem
\ref{MSF}) for minimal spanning forests on slabs and sketch our proof. Section
\ref{sec:RSW} is devoted to the proof of Russo-Seymour-Welsh type box-crossing
theorems on slabs, which are used in our argument and are also of interest in
their own right (see especially Theorem~\ref{thm:BXP} and Corollary~\ref{cor}).
Finally, in Section \ref%
{sec:pf}, we collect all the ingredients to prove our gluing lemma for invasion
clusters, and thus conclude the proof of Theorem~\ref{MSF}. We note that one
ingredient, Lemma \ref{combi}, is an extension of the combinatorial lemma of
\cite{DST}. The extension is needed for the invasion setting where continuous
edge variables replace Bernoulli ones.

\section{Background and First Main Result\label{sec:back}}

\subsection{Minimal Spanning Forests\label{sec:MSF}}
Let $G=\left( V,E\right) $ be a finite graph. A subgraph $H$ of $G$ is spanning
if $H$ contains all vertices of $G$. A labeling is an injective
function $\omega :E\rightarrow %
\left[ 0,1\right] $. The number $\omega _{e}\dot{=%
}\omega \left( e\right) $ will be referred to as the label of $e$. Note that the
labeling induces a total ordering on $E$, where $e\prec e^{\prime }$ if $%
\omega \left( e\right) <\omega \left( e^{\prime }\right) $.

Define $\mathcal{T}^{\omega }$ to be a spanning subgraph of $G$ whose edge
set consists of all $e\in E$ whose endpoints cannot be joined by a path
whose edges are all strictly smaller than $e$. It is easy to see that $\mathcal{T}%
^{\omega }$ is a spanning tree, and in fact, among all spanning trees $%
\mathcal{T}$, $\mathcal{T}^{\omega }$ minimizes $\sum_{e\in \mathcal{T}%
}\omega \left( e\right) $ (\cite{LPS}).

\begin{definition}
  When $\left\{ \omega \left( e\right) :e\in E\right\} $ are i.i.d.\@ uniform $%
  \left[ 0,1\right] $ random variables, the law of the corresponding spanning
  tree $\mathcal{T}^{\omega }$ is called the minimal spanning tree (MST). The
  law of $\mathcal{T}^{\omega }$ defines a probability measure on $2^{E}$ (where we
  identify the tree $\mathcal{T}^{\omega }$ with its set of edges).
\end{definition}

When passing to infinite graphs, two natural definitions of minimal spanning
forests can be made, that arise as weak limits of minimal spanning trees on
finite graphs.

Let $G=\left( V,E\right) $ be an infinite graph, and $\omega :E\rightarrow %
\left[ 0,1\right] $ be a labeling function. Let $\mathcal{F}_{\mathrm f}^{\omega }$
be the set of edges $e\in E$, such that in every path in $G$ connecting the
endpoints of $e$, there is at least one edge $e^{\prime }$ with $\omega
\left( e^{\prime }\right) >\omega \left( e\right) $. When $\left\{ \omega
\left( e\right) :e\in E\right\} $ are i.i.d.\@ uniform $\left[ 0,1\right] $
random variables, the law of $\mathcal{F}_{\mathrm f}^{\omega }$ is called the free
minimal spanning forest (FMSF) on $G$.

An extended path joining two vertices $a,b\in V$ is either a simple path in $%
G$ joining them, or the disjoint union of two simple semi-infinite paths,
starting at $a,b$ respectively. Let $\mathcal{F}_{\mathrm w}^{\omega }$ be the set
of edges $e\in E$, such that in every extended path in $G$ connecting the
endpoints of $e$, there is at least one edge $e^{\prime }$ with $\omega
\left( e^{\prime }\right) >\omega \left( e\right) $. Analogously, when $%
\left\{ \omega \left( e\right) :e\in E\right\} $ are i.i.d. uniform $\left[
0,1\right] $ random variables, the law of $\mathcal{F}_{\mathrm w}^{\omega }$ is
called the wired minimal spanning forest (WMSF) on $G$.

It is clear that $\mathcal{F}_{\mathrm f}^{\omega }$ and $\mathcal{F}_{\mathrm w}^{\omega }$
are indeed forests. In addition, all the connected components in $\mathcal{F%
}_{\mathrm f}^{\omega }$ and $\mathcal{F}_{\mathrm w}^{\omega }$ are infinite. In fact, the
smallest label edge joining any finite vertex set to its complement belongs to
both forests.

We now describe how $\mathcal{F}_{\mathrm f}^{\omega }$ and $\mathcal{F}_{\mathrm w}^{\omega
}$ arise as weak limits of minimal spanning trees on finite graphs. Consider
an increasing sequence of finite, connected induced subgraphs $G_{n}\subset
G $, such that $\cup _{n\geq 1}G_{n}=G$. For $n\in \mathbb{N}$, let $%
G_{n}^{\mathrm w}$ be the graph obtained from $G$ by identifying the vertices in $%
G\backslash G_{n}$ to a single vertex.

\begin{proposition}[\protect\cite{Alex}\cite{LPS}]
Let $\mathcal{T}_{n}^{\omega },\bar{\mathcal T}_{n}^{\omega }
$ denote the minimal spanning tree on $G_{n}$ and $G_{n}^{\mathrm w}$, respectively,
that are induced by the labeling $\omega $. Then for any labeling function $\omega$, \[ \mathcal{F}_{\mathrm f}^{\omega
}=\lim_{n\rightarrow \infty }\mathcal{T}_{n}^{\omega }\text{, and } \mathcal{F}
_{\mathrm w}^{\omega }=\lim_{n\rightarrow \infty }\bar{\mathcal T}_{n}^{\omega }. \]
This means for every $e \in \mathcal{F}_{\mathrm f}^{\omega}$, we have $e \in \mathcal{T}_{n}^{\omega }$
for all sufficiently large $n$, and similarly for $\mathcal{F}_{\mathrm w}^{\omega}$.
\end{proposition}

One natural question on a given connected graph is whether the free and
wired minimal spanning forests coincide. To answer this question, we need to
explain the relation to critical Bernoulli percolation.

\begin{proposition}[\protect\cite{LPS},\cite{Alex}]
\label{unique}On any connected graph $G$, we have $\mathcal{F}_{\mathrm f}^{\omega }=%
\mathcal{F}_{\mathrm w}^{\omega }$ if and only if for almost every $p\in \left( 0,1\right) $,
 Bernoulli percolation on $G$ with parameter $p$ has at most one infinite
cluster a.s.
\end{proposition}

Let $[k]:=\left\{ 0,...,k\right\} $ and $\mathbb S_{k}:=\mathbb{Z}%
^{2}\times \lbrack k]$ be the slab of thickness $k$. It follows from \cite%
{AKN} and \cite{BK} that the infinite cluster on $\mathbb{S}_{k},$ if it exists,
is a.s.\@ unique. Therefore on $\mathbb{S}_{k}$, WMSF and FMSF coincide. This
justifies referring to \emph{the} minimal spanning forest on $\mathbb S_k$
without ambiguity.

\subsection{Invasion Percolation} \label{inv}

We now define invasion percolation, an object closely related to WMSF and
critical Bernoulli percolation. Let $\{ \omega ( e) :e\in
  E\} $ be i.i.d.\@ uniform $[ 0,1] $ random variables. The
invasion cluster $\mathcal{I}_{v}$ of a vertex $v$ is defined as a union of
subgraphs $\mathcal{I}_{v}( k) $, where $\mathcal{I}_{v}(
  0) =\{ v\} $, and $%
\mathcal{I}_{v}( k+1) $ is $\mathcal{I}_{v}( k) $ together
with the lowest labeled edge (and its vertices) not in $\mathcal{I}_{v}( k) $ but incident to
some vertex in $\mathcal{I}_{v}( k) $.

We also define the invasion tree, $T_{v}$ of a vertex $v$, as the increasing
union of trees $T_{v}( k) $, where $T_{v}( 0) =\{
v\} $, and $T_{v}( k+1) $ is $T_{v}( k) $
together with the lowest edge (and its vertices) joining $T_{v}( k) $ to a vertex
{\it not\/} in $T_{v}( k) $. Notice that $\mathcal{I}_{v}$ has the same
vertices as $T_{v}$, but may have additional edges.

The following proposition in \cite{LPS} (see also \cite{NS}) describes the relation between
invasion trees and WMSF.

\begin{proposition} \label{invtree} Let $G=( V,E) $ be a locally
  finite graph. Then
\begin{equation*}
\mathcal{F}_{\mathrm w}^{\omega }=\cup _{v\in V}T_{v} \quad\text{a.s.}
\end{equation*}
\end{proposition}

Therefore, to show $\mathcal{F}_{\mathrm w}^{\omega }$ is a single tree, it suffices
to prove for any $v\in V$, $\mathcal{I}_{0}\cap \mathcal{I}_{v}\neq
\emptyset $.

We now describe the connection between invasion percolation and critical Bernoulli
percolation. An edge $e\in E$ is said to be $p$-open if its weight satisfies
$\omega(e)<p$. The connected components of the graph induced by the $p$-open
edges are called $p$-open clusters. Notice that the set of $p$-open edges is a
Bernoulli bond percolation process on $G$ with edge density~$p$.

Let $p_{c}( G) $ be the critical
probability for Bernoulli bond percolation on $G$. For any $p>p_{c}( G) $,
 there exists almost surely an infinite $p$-open cluster. Suppose that
for some $k$, $\mathcal{I}_{v}( k) $ contains a vertex of this
cluster. Then all edges invaded after time $k$ remain in this cluster.

To make another observation, denote by $\mathcal{C}%
_{p_{c}}( v) $ the $p_{c}$-open cluster of a vertex $v\in G$, and
write $\theta_v ( p_c)$ for the probability that $\mathcal
C_{p_c}(v)$ is infinite. If $\theta_v ( p_c) =0$, then of course
the $p_c$-open cluster $\mathcal{C}_{p_{c}}( v) $ is finite a.s. This
implies that once $v$ is reached by an invasion, then (with probability $1$) all edges in $\mathcal{C}_{p_{c}}(
  v) $ will be invaded before any edges with label $\geq p_{c}$ are invaded.
In particular, when $\theta_v(p_c)=0$, the $p_c$-cluster of $v$
satisfies $\mathcal{C}_{p_{c}}( v)\subset \mathcal{I}_{v} $.

\subsection{Notation, conventions\label{sec:notation-conventions}}

\bigskip We consider the space of configurations $( \Omega
  ,\mathcal{F},\mathbb{P}) $, where $\Omega =[ 0,1] ^{E}$ ($E$
denotes the edge set of the slab $\mathbb{S}_{k}$, given by the pairs of points
at Euclidean distance 1 from each other), $\mathcal{F}$ is the Borel $\sigma $-%
field on $\Omega $, and $\mathbb{P}$ is the underlying (product of uniforms)
probability measure. Given the labelling function $\omega \in \Omega $, and
$S\subset E$, we use $%
\omega |_{S}$ to denote the restriction of $\omega $ to $S$.

Given $a,b\in \mathbb{Z}$, $a<b$, let $[ a,b] =\{
  a,a+1,...,b\} $, and we simply denote by $[k]$ the set $[
  0,k] $. For any subset $S\subset \mathbb{Z}^{2}$, we denote by
$\bar{S}$ the set $S\times [ k] \subset \mathbb{S}_{k}$, for $z \in
\mathbb{Z}^{2}$, we denote by $\bar{z}$ the set $\{z\}\times [k]$, and for
$S\subset \mathbb S_k$, we denote by $\bar S$ the set $\pi(S)\times[k]$, where
$\pi(S)$ is the projection of $S$ onto $\mathbb Z^2$. For $x\in \mathbb{S}_{k}$
and $U,V\subset $ $\mathbb{S}_{k}$, we denote by $\vert x\vert $ the
Euclidean norm of $x$, and $\mathrm{dist}( U,V) \dot{=}\min_{u\in U,v\in
  V}\vert u-v\vert $. For $x=(x_{1}, x_{2}, x_{3}), y=(y_{1}, y_{2},
y_{3}) \in \mathbb S_k$ and $U,V\subset \mathbb{S}_{k}$, we define
$\mathrm{dist}^{*}(x,y) =\max \{|x_{1}-x_{2}|, |y_{1}-y_{2}|\}$, and $\mathrm{dist}^{*}(U,V)
=\min_{x \in U, y\in V} \text{dist}^{*}(x,y)$. The vertex-boundary of a set
$U\subset \mathbb S_k$ is denoted by $\partial U$ (it is defined as the set of
vertices in $U$ with a neighbor in $\mathbb S_k\setminus U$). Given $%
m>n>0$, define $B_{n}(x) =x+ [ -n,n] ^{2}$ and $%
A_{n,m}( x) =B_{m}( x) \backslash B_{n}( x) $.
When $x$ is the origin, we will omit the dependence on $x$. If $z\in \mathbb
S_k$, we write $B_n(z)$ for the ball $B_n(\pi(z))$, where $\pi(z)$ is the projection of $z$
onto $\mathbb Z^2$.

\subsection{Single Tree Result}

\begin{theorem}
\label{MSF}For any $k\in \mathbb{N}$, the minimal spanning forest on $%
\mathbb{S}_{k}$ is a single tree a.s.
\end{theorem}

\begin{remark}
As we will see from the proof below, the same argument applies to
$\mathbb{Z}^{2}\times F$, where $F$ is any finite connected graph. This includes $F=\{
0,...,k\} ^{d-2}$, for $d\geq 3$. Similar arguments also apply to the finite range extensions
$\mathbb{Z}^{2}_{K} = (\mathbb{Z}^{2}, E_{K})$ of $\mathbb{Z}^{2}$, where $E_{K}= \{(x,y):|x-y|\leq K\} $.
\end{remark}

A sketch of the proof is as follows; the complete proof is in Sections \ref{sec:RSW} and \ref{sec:pf} below. By Proposition \ref{invtree}, it suffices to prove that for any $x\in \mathbb{S}%
_{k}$, $\mathcal{I}_{0}\cap \mathcal{I}_{x}\neq \emptyset $. This is shown in two steps.

1. We first prove that with bounded away from zero probability, there is a
$p_{c}$-open circuit in the annulus $\ol{A_{n,2n}}$. This follows from the
box-crossing property for critical Bernoulli percolation on $\mathbb{S}_{k}$
which we will prove in Section \ref{sec:RSW} --- see Theorem \ref{thm:BXP}. The
proof uses a new Russo-Seymour-Welsh type theorem, based on gluing lemmas given
in Section \ref{sec:RSW} below (like those first established in \cite{DST}).

2. It follows from Step 1 that infinitely many disjoint $p_{c}$-open circuits in
$\mathbb{S}_{k}$ ``surround'' the origin. In particular, the projections of
$\mathcal{I}_{0}$ and $\mathcal{I}_{x}$ on $\mathbb{Z}^{2}$ intersect the
projections of $p_{c}$-open circuits infinitely many times. In
Section~\ref{sec:pf}, we prove a version of the gluing lemma adapted to invasion
clusters, which says, roughly speaking, that each time the invasion cluster
$\mathcal{I}_{0}$ crosses an annulus $\ol{A_{n,2n}}$, it ``glues'' to a
$p_c$-open circuit $\mathcal C_n$ in this annulus with probability larger than a
constant $c>0$ (independent of $n$). The same argument applies to the invasion
cluster $\mathcal{I}_{x}$. Therefore, with probability 1, $\mathcal{I}_{0}$ and
$\mathcal{I}_x$ eventually are glued to the same $p_c$-open circuit, which
implies that $\mathcal{I}_{0}\cap \mathcal{I}_{x}\neq \emptyset$.


\section{RSW Theory and Power Law Decay on Slabs\label{sec:RSW}}

In this section, we consider Bernoulli bond percolation with density $p$ on the
slab~$\mathbb{S}_{k}$: each edge is declared independently open with probability
$p$ and closed otherwise. Write $\mathbf P_p$ for the resulting probability
measure on the configuration space $\{0,1\}^{E}$. Let $R=\ol{[x,x^{\prime }]\times
\lbrack y,y^{\prime }]}$ be a rectangle in $\mathbb S_k$. We say that $%
R$ is crossed horizontally if there exists an open path from $\ol{\{x\}\times
[y,y^{\prime }]}$ to $\ol{\{x^{\prime }\}\times [y,y^{\prime }]}$ inside $R$. We denote this
event by $\mathcal{H}(R)$. For $m,n\geq 1$, we define for $p \in [0,1]$,
\begin{equation}
f(m,n)=f_p(m,n):=\mathbf{P}_{p}\left[ \mathcal{H}\left (\ol{[0,m]\times[0,n]}\right)\right].
\end{equation}

In this section, we prove the following result, that the box-crossing property holds
for critical Bernoulli percolation on the slab $\mathbb S_k$, for every fixed
$k\ge 1$.

\begin{theorem}[Box-crossing property]
\label{thm:BXP} Let $p=p_{c}(\mathbb S_{k})$. For every $\rho >0$, there
exists a constant $c_{\rho }$ such that for every $n\geq 1/\rho $,
\begin{equation}
c_{\rho }\leq f(n,\lfloor \rho n\rfloor )\leq 1-c_{\rho }.  \label{eq:18}
\end{equation}
\end{theorem}

\begin{remark}
  In our proof, the bound $c_\rho$ we obtain depends on the thickness $k$ of the
  slab. Due to our use of the gluing lemma, the bounds we obtain get worse when
  the thickness of the slab increases. More precisely, for fixed $\rho>0$ and
  for the slab with thickness $k$, our proof provides us with a constant
  $c_\rho=c_\rho(k)$,  and the sequence $(c_\rho(k))$ converges quickly to $0$ as the
  thickness $k$ tends to infinity. Getting better bounds would be very
  interesting to help understand critical behavior on $\mathbb Z^3$
  (which corresponds to $k=\infty$).
\end{remark}

The box-crossing property was established by Kesten
(see~\cite{kesten1982percolation}) for critical Bernoulli percolation on two
dimensional lattices under a symmetry assumption. The proof relies on a result
of Russo, Seymour and Welsh (\cite{russo1978note,seymour1978percolation}) which
relates crossing probabilities for rectangles with different aspect ratios.
Recently, the box-crossing property has been extended to planar percolation
processes with spatial dependencies, e.g. continuum
percolation~\cite{tassion2015crossing,alhberg2015continuum} or the
random-cluster model \cite{duminil2015continuity}.

The box-crossing property has been instrumental in many works on Bernoulli
percolation, and has numerous applications. These include Kesten's scaling
relations \cite{kesten1987scaling}, bounds on critical exponents (e.g.
polynomial bounds on the one arm event), computation of universal
exponents and tightness arguments in the study of the scaling limit
\cite{smirnov2001critical}, to name a few. We expect that similar results can be
derived from the box-crossing property of Theorem~\ref{thm:BXP}. Next, we state
some direct useful consequences of the box-crossing property (these are proved in Section
\ref{sec:proof-coroll-prot}).

\begin{corollary}
\label{cor}For critical Bernoulli percolation on the slab $\mathbb S_k$, we have:
\begin{enumerate}[\bf 1.]
\item\label{item:3} {\normalfont [Existence of circuits with positive probability]} There
  exists $c>0$ such that for every $n\ge1$,
  \[\mathbf P_p [ \text{there exists an open circuit in $\ol{A_{n,2n}}$ surrounding $\ol{B_n}$} ]\ge c .\]

\item\label{item:4} {\normalfont [Existence of blocking surfaces with positive probability]} There exists
 $c>0$ such that for every $n\ge1$,
  \[\mathbf P_p [ \text{there exists an open path from $\bar{B_n}$  to $\partial
  \bar{B}_{2n}$} ]\le 1-c .\]

\item\label{item:5} {\normalfont [Polynomial decay of the 1-arm event]}  There exists $\delta >0$, such
  that for $n>m\geq 1$,
  \begin{equation*}
    \mathbf P_p [ \text{there exists an open path from $\bar{B}_{m}$ to $\partial \bar{B}_{n}$} ] \leq \left( m/n\right) ^{\delta }.
  \end{equation*}
\end{enumerate}
\end{corollary}

\smallskip
\noindent\textit{Remarks.}

  \begin{enumerate}[1.]
  \item Item~\ref{item:4} can be interpreted geometrically via duality: there is
    no open path from $\bar B_n$ to $\partial \bar B_{2n}$ if and only if there
    exists a blocking closed surface in the annulus $\ol{A_{n,2n}}$, made up of
    the plaquettes in the dual lattice perpendicular to the closed edges in
    $\mathbb S_k$. Therefore, Item~\ref{item:4} can be understood as the
    existence with probability at least $c$ of such a blocking surface in the
    annulus $\ol{A_{n,2n}}$.
  \item Item~\ref{item:5} implies in particular that critical percolation on the
    slab $\mathbb S_k$ does not have an infinite cluster. It strengthens the
    previous result of \cite{DST}. Moreover, our proof leads to the bound
    $\delta \ge C^{-k}$ for some $C < \infty$. Strengthening the lower bound of
    $\delta(k)$ would also be very interesting.
  \end{enumerate}

  Planar geometry is a key ingredient for the proofs of the existing box-crossing results on planar graphs. In the case
  of non-planar graphs, one side of the inequality can still be proved by
  standard renormalization arguments. Namely, the crossing probability of the
  short side of the rectangle is bounded from below. This is sketched as Lemma \ref{lem:FC2} in Section
  \ref%
  {sec:renorm-inputs}. The more difficult part is to carry out a renormalization argument to prove the crossing probability of the long side
is bounded from above (this is done in Lemma \ref{lem:FC1}), and to relate the crossing probability
  of the long side to that of the short side (done in Section  \ref{sec:RSWpf}). This is where the quasi-planarity of the slabs comes into play. For
  planar graphs such relations can be obtained by repeated use of the Harris-FKG
  inequality and one of its consequence known as the square root trick. For
  non-planar graphs, two paths may not intersect even if their projections on
  the plane do intersect. In Sections \ref{sec:gluing-lemma} we prove versions of gluing lemmas for open paths and circuits in $\mathbb{S}_{k}$.
 In Section \ref{sec:defin-dual-surf} we apply the gluing lemmas to bound the crossing probability for rectangles with different aspect ratio. Finally, we put these ingredients together
  and complete the proof of RSW type theorems in Section \ref{sec:RSWpf}, \ref{sec:rsw-theorem:-high} and \ref{sec:proof-theorem1}.

\subsection{Positive correlation and the square-root trick}
\label{sec:posit-corr-square}

In this section we recall the Harris-FKG inequality about positive
correlation of increasing events, and an important consequence called the
square-root trick. We refer to \cite{grimmett1999} for more details. A
percolation event $\mathcal A\subset \{0,1\}^E$ is said to be increasing if
\begin{equation}
  \label{eq:42}
  \left. \begin{array}{cl}
    \omega \in\mathcal A\\
    \forall e\in E,\, \omega(e)\le \eta(e)
  \end{array}
  \right\} \implies \eta\in\mathcal A.
\end{equation}
It is decreasing if $\mathcal A^c$ is increasing.

\begin{theorem}[Harris-FKG inequality]\label{thm:FKG}
  Let $p\in [0,1]$. Let $\mathcal A,\mathcal B$ be two increasing events (or two
  decreasing events), then
  \begin{equation}
    \label{eq:46}
    \Pp{\mathcal A\cap \mathcal B}\ge \Pp{\mathcal A} \Pp{\mathcal B}.
  \end{equation}
\end{theorem}

The following straightforward consequence, called the square-root trick, will be
very useful.
\begin{corollary}[Square-root trick]\label{cor:SRT}
  Let $\mathcal A_1,\ldots,\mathcal A_j$ be $j$ increasing events, then
  \begin{equation}
    \label{eq:47}
    \max\{\Pp{\cal A_1},\ldots,\Pp{\mathcal A_j}\} \ge 1-\left(1-\Pp{\cal
      A_1\cup\cdots\cup\cal A_j}\right)^{1/j}.
  \end{equation}
\end{corollary}

\subsection{Gluing Lemmas}

\label{sec:gluing-lemma}

Define $\mathcal{H}$ to be the set of continuous and
strictly increasing functions $h:[0,1]\rightarrow \lbrack 0,1]$. Clearly,
given $h_{1},h_{2}\in \mathcal{H}$, we have $h_{1}h_{2}\in \mathcal{H}$, $%
h_{1}^{-1}\in \mathcal{H}$, and $1-h_{1}^{-1}\left( 1-\cdot \right) \in
\mathcal{H}$. We sometimes denote by $\mathsf{h}$ a function in $\mathcal{H}$ that may change from line to line.

In order to state the gluing lemmas we need to fix an ordering $\prec $ on the
vertices of $\mathbb{S}_{k}$. The choice of the ordering is flexible; ours
is the following. Given $x,y\in \mathbb{S}_{k}$, we write $x\prec
y $ iff

\begin{itemize}
\item $\left\vert x\right\vert <\left\vert y\right\vert $, or

\item $\left\vert x\right\vert =\left\vert y\right\vert $, and there exists $%
k$ such that $x_{i}=y_{i}$ for $i<k$, and $x_{k}<y_{k}$.
\end{itemize}

We order directed edges and more generally, site self-avoiding paths of
$\mathbb{S}_{k}$ by taking the corresponding lexicographical order, as in
Section 2.3 of \cite{DST}. Let $S$ be a connected subset of $\mathbb S_k$. For
$A,B,S\subset \mathbb S_k$, the event $A \lr[S] B$ denotes the existence of a
path of open edges in $S$ connecting $A\cap S$ to $B\cap S$. If this event
occurs, define $\Gamma _{\mathrm{min}}^S(A,B)$ to be the minimal (for the order
defined above) open self-avoiding path in $S$ from $A$ to $B$. Set
$\Gamma^S_{\mathrm{min}}(A,B)=\emptyset$ if there is no open path from $A$ to
$B$ in $S$. Note that $\Gamma^S _{\mathrm{min}}(A,B)$ is defined relative to a
fixed $S$. Sometimes we will also use the definitions above with $A$ and $B$
random sets (they may depend on the configuration $\omega$).

We will repeatedly use the following combinatorial lemma stated in \cite{DST}.

\begin{lemma}
\label{combi0}
Let $s,t>0$. Consider two events $\cal A$ and $\cal B$ and a map $\Phi$ from
$\cal A$ into the set $\mathfrak{P}(\cal B)$ of subevents of $\cal B$. We assume
that:
  \begin{enumerate}
  \item for all $\omega\in \cal A$, $|\Phi(\omega)|\geq t$,
  \item for all $\omega' \in \cal B$, there exists a set $S$ with less than
  $s$ edges such that $\{\omega: \omega' \in \Phi(\omega)
    \} \subset \{\omega : \omega_{|_{S^c}}=\omega'_{|_{S^c}}\}$.
  \end{enumerate}
  Then,
  \begin{equation}
     \Pp{\cal A}\leq\frac1t \left(\frac2{\min\{p,1-p\}}\right)^s \P{\cal B}.
  \end{equation}
\end{lemma}

We now state two gluing lemmas for open paths crossing subsets of rectangular regions
 of the form $R\times[k]$, with $R$ a topological rectangle. The first, Theorem \ref{GL0}, has the simplest geometry
and will be proved as a consequence of essentially the same arguments used for Theorem \ref{GL}, which has a more complicated geometry.

As in \cite{DST}, the proof of these gluing lemmas uses local modifications of
percolation configurations, which rely on the following definition.
\begin{definition}\label{def:rad1}
  Define an integer $r\geq 3$ such that for every $s\ge r $ and every $z\in L
  \dot{=} \mathbb{Z}_{+}^{2}\times [k] \setminus \{(0,0,0),(0,0,k)\} $, the following holds.
 For any three distinct neighbors $u,v,w$ of $z$, and any three distinct
    sites $u',v',w'$ (that are also distinct from $u,v,w$) on the boundary of $%
    \ol{B_{s}\left( z\right)} \cap L$, there exist three disjoint self-avoiding
    paths in $\ol{B_{s}\left( z\right)} \cap L\backslash \left\{ z\right\} $
    connecting $u$ to $u^{\prime }$, $v$ to $v^{\prime }$ and $w$ to
    $w^{\prime}$.
  \end{definition}
In the case of slabs with $k\geq 1$, it suffices to take $r=3$. We will present
the proof with general $r$, since that can be adapted to the more general
quasi-planar graphs $\mathbb{Z}^{2} \times F$, with $F$ a finite connected
graph.

\begin{theorem}
  \label{GL0} Let $r\ge3$ be as in Definition~\ref{def:rad1}. Fix $\eps>0$ and
  $k \geq 1$. There exists $\mathsf{h_{0}}\in \mathcal{H}$ such that the
  following holds. Let $S$ be a subset of $\mathds S_{k}$ of the form
  $\ol{[a,b]\times[c,d]}$, with $b-a \geq r+2$, $d-c \geq r+2$. Let $A,B,C,D$ be
  four subsets of $\partial S$ such that their projections on $\mathds{Z}^{2}$
  are disjoint, and such that the projection on $%
  \mathds{Z}^{2}$ of any path from $A$ to $B$ in $S$ intersects the projection
  of any path from $C$ to $D$ in $S$. Then for every $p\in[\eps,1-\eps]$,
\begin{equation}
\label{eq:190}
       \Pp{C\lr[S] A} \ge \mathsf{h}_0( \Pp{A\lr[S] B}\wedge   \Pp{C\lr[S]D}).
  \end{equation}
\end{theorem}

Roughly speaking, when both $A\lr[S] B$ and $C\lr[S]D$ occur with uniformly
positive probability, so does $C\lr[S] A$. If both $A\lr[S] B$ and $C\lr[S]D$
occur with high probability, then so does $C\lr[S] A$.

The proof of Theorem \ref{GL0} (see below) is a slightly modified version of the proof of the following theorem.

\begin{theorem}[Main gluing lemma for paths]
\label{GL}
Let $r\ge 3$ be  as in Definition~\ref{def:rad1}. Fix $\eps>0$ and $k\geq 1$. There
exists $\mathsf{h_{0}}\in \mathcal{H}$ such that the following holds. Let $S, R$ be two subsets of $\mathbb S_k$ of the form $\ol{[a,b]\times[c,d]}$,
with $b-a\geq r+2$, $d-c\geq r+2$. Let $A,B\subset \partial S$ and $C\subset
\partial R$ be such that the projections of $A,B,C$ on $\mathds{Z}^{2}$ are at least sup-norm distance $r+2$ apart from each other. Then for every $p\in[\eps,1-\eps]$,
\begin{equation}
    \label{eq:19}
    \Pp{C\lr[R] A}\ge \mathsf{h}_0( \Pp{C\lr[R] \mathcal{N}(\bar \Gamma, r)}),
  \end{equation}
where $\Gamma=\Gamma_{\mathrm{min}}^S(A,B)$ and $\mathcal{N}(\bar \Gamma, r) = \{ x \in S : \mathrm{dist}^{*}(x, \Gamma) \leq r\}$.
\end{theorem}

\noindent\textit{Remarks.}
\begin{enumerate}[1.]
\item We note that in the simpler case of the plane (when $k = 0$), the FKG
  inequality implies that the left hand side of \eqref{eq:190} is greater than
  or equal to
\[ \mathbf{P}_{p}\left[ A\lr[S]B\right] \mathbf{P}_{p}\left[  C\lr[S]D\right], \]
 and thus
  \eqref{eq:190} is valid with $\mathsf{h}_0(x) = x^{2}$.
\item\label{item:6} For better readability of the proof we have not presented \mbox{Theorem
    \ref{GL}} in the highest level of generality. In particular the sets $R$ and
  $S$ in the statements of Theorem~\ref{GL0} and Theorem~\ref{GL} do not need
  to be rectangles. The proof also applies if the sets $R$ and $S$ are of the
  form $\ol{T}$, where $T\subset \mathbb Z^2$ is a rectilinear domain such that $\partial T$ is a simple circuit made of 
 vertical/horizontal segments of length at least $ r+2$, and such that any two disjoint segments are at least sup-norm 
distance $r+2$ apart from each other.

\item In the proof we will see that the function $\mathsf h_0\in\mathcal H$ can be chosen in
  such a way that $\mathsf h_0(x)\ge c_0 x$, where $c_0$ is a constant that
  depends only on $\eps$ and $k$. This remark will be important in the
  proof of Theorem~\ref{cor:GL}.
\end{enumerate}

\begin{proof} [Proof of Theorem \ref{GL}] The proof consists of two parts. In
  the first part, we prove that for some $\delta >0 $, there is an
  $h_{1}:\left[0 ,1\right] \rightarrow [0,1]$ that is continuous on $[0,1]$,
  strictly increasing on $[1-\delta,1]$, and with $h_{1}\left( 1\right) =1$, such
  that (\ref{eq:19}) is valid with $\mathsf{h}_0$ replaced by $h_{1}$. In the
  second part, we show that (\ref{eq:19}) is valid with $\mathsf{h}_0$ replaced
  by $h_{2}\left( x\right) =c_{0}x$, $x\in \left[ 0,1\right] $, for some
  $c_{0}>0$ (which depends only on $\eps$ and $k$). Therefore one can take
  $\mathsf{h_{0}}\left( x\right) =\max \left\{ h_{1}\left( x\right)
    ,c_{0}x\right\} $, which is indeed in $\mathcal{H}.$ The proof of the
  first part is very similar to the argument in Section 2.3 of \cite{DST}, and
  we next outline those arguments, with details supplied to show that $h_{1}$ is
  continuous and strictly increasing.

  For the proof below we will slightly abuse notation and use
  $\ol{B_{s}(z)}$ to denote $\ol{B_{s}(z)} \cap S$.

  Following Section 2.3 of \cite{DST}, we define $U(\omega)$, $%
  \omega \in \left[ 0,1\right] ^{\mathbb{S}_{k}}$ to be the set of vertices $z
  \in S$ such that

\begin{itemize}
\item $z \in \bar \Gamma$, and
\item $\overline{B_{r+1}\left( z\right) }$ is connected to $C$ in $R$
by an open path $\pi $, such that $\mathrm{dist}^{*}(\pi, \bar \Gamma)= r+1$.
\end{itemize}

We discuss two different cases below, depending on the cardinality of $U(\omega)$. We will use the following two events:
\begin{equation*}
\mathcal{X}=\left\{ C\lr[R] \mathcal{N}(\bar \Gamma, r)\right\} \cap \left\{ C\lr[R] A \right\} ^{c}
\end{equation*}%
and
\begin{equation*}
\mathcal{X}^{\prime }=\left\{ C\lr[R] \mathcal{N}(\bar \Gamma, r)\right\} \cap \left\{ C\lr[R]A \right\} .
\end{equation*}
Our object is basically to show that $\Pp{\mathcal{X}}/\Pp{\mathcal{X}^{\prime}}$ is small, at least when $\Pp{\mathcal{X}}+ \Pp{\mathcal{X}^{\prime}} = \Pp{C\lr[R] \mathcal{N}(\bar \Gamma, r)}$ is not small.

\textbf{Fact 1.} There exists $C_{1}<\infty $, depending only on $\eps$, such
that for any $t >0$,
\begin{equation*}
\mathbf{P}_{p}\left[ \mathcal{X}\cap \left\{ \left\vert U\right\vert
\leq t\right\} \right] \leq (C_{1})^{\/t}\mathbf{P}_{p}\left[ \left(C\lr[R] \mathcal{N}(\bar \Gamma, r)\right) ^{c}\right] .
\end{equation*}

We prove this statement by constructing a disconnecting (or anti-gluing) map
\begin{equation*}
\Phi :\mathcal{X}\cap \left\{ \left\vert U\right\vert \leq t\right\} \rightarrow
\left\{C\lr[R] \mathcal{N}(\bar \Gamma, r)\right\} ^{c},
\end{equation*}%
such that for any $\omega ^{\prime }$ in the image of $\Phi $, the
cardinality of its pre-image is bounded by a constant depending only on $t$.

We define $\Phi(\omega) $ by closing for every $z\in U( \omega)$, all the edges
adjacent to a vertex in $\overline{B_{r}(z)}$ which are not in $\Gamma$. Observe
that $\Phi(\omega)$ cannot contain any open path from $C$ to $\mathcal{N}(\bar
\Gamma, r)$. Let $\vert \ol{B_{r+1}}\vert $ denote the number of edges in
$\ol{B_{r+1}}$. Lemma~\ref{combi0} can be applied with $s= 2 t \vert
\ol{B_{r+1}} \vert $ to yield

\begin{equation*}
\mathbf{P}_{p}\left[ \mathcal{X}\cap \left\{ \left\vert U\right\vert
\leq t\right\} \right] \leq \left( \frac{2}{p\wedge \left( 1-p\right) }\right)
^{2t\vert\ol{ B_{r+1}}\vert }\mathbf{P}_{p}\left[ \left( C\lr[R] \mathcal{N}(\bar \Gamma, r)\right) ^{c}\right] ,
\end{equation*}%
and we can conclude the proof of Fact 1 with $C_{1}=(2/\eps)^{2\vert \ol{B_{r+1}} \vert}$.

\textbf{Fact 2.} There exists $C_{2}<\infty $, depending only on $\eps$ and $k$, such that for any $t>8$,
\begin{equation*}
\mathbf{P}_{p}\left[ \mathcal{X}\cap \left\{ \left\vert U\right\vert
>t\right\} \right] \leq \frac{C_{2}}{t-8}\mathbf{P}_{p}\left[ \mathcal{X}%
^{\prime }\right].
\end{equation*}

We prove Fact 2 by constructing a map
\begin{equation*}
\Phi :\mathcal{X}\cap \left\{ \left\vert U\right\vert >t\right\} \rightarrow
\mathfrak{P}(\mathcal{X}^{\prime }),
\end{equation*}%
such that for any $\omega ^{\prime } \in \mathcal{X}^{\prime }$, the $\omega$\text{'s} with
$\omega ^{\prime } \in \Phi (\omega)$ agree on all but at most $s$ specified edges, with $s$
a constant depending only on $k$.

The construction of $\Phi $ is similar to that in \cite{DST}, as we now
describe. For any $z\in U(\omega)$ that is not one of the eight corners of $S$, we will construct a new
configuration $\omega ^{\left( z\right) }$ and define $\Phi (\omega) = \{ \omega^{(z)}:
z \in U(\omega)\text{, $z$ is not a corner of } S\}$. The new configuration $\omega^{(z)}$ is constructed by the following three steps.

\begin{enumerate}
\item Define $u^{\prime },v^{\prime }$
to be respectively the first and last vertices (when going from $A$ to $B$)
of $\Gamma(\omega)$ which are in $\overline{B_{r+1}\left(
z\right) }$. Choose $w^{\prime }$ on the boundary of $%
\overline{B_{r+1}\left( z\right) }$, such that there exists an open self-avoiding
path $\pi$ (which could be a singleton) from $w^{\prime }$ to $C$. By the definition of $U\left( \omega \right) $ and $%
\mathcal{X}$, $w^{\prime }$ is distinct from $u^{\prime }$, $z$
and $v^{\prime }$. Choose $u,v,w$ such that $\left( z,u\right) ,\left( z,v\right)$ and
$\left( z,w\right) $ are three distinct edges with $%
v \prec w $. If $z =u$ or $z$ is a neighbor of $u$, we simply take $u = u'$, and if $z =v$, we take
$v=v'=z$. Otherwise, $u,v,w$ are chosen to be distinct sites from $u', v',w'$.
(Note that this is possible
because for $z$ that is not a corner of $S$, the degree of $z$ is at least $4$. And since $A, B,C$ are at least distance
$r+2$ apart, at most one of $u', v',w'$ can be a neighbor of $z$).

\item Close all edges of $\omega $ in $\overline{B_{r+2}\left( z\right) }$ except the edges of $\overline{B_{r+2}\left( z\right) }\backslash
\overline{B_{r+1}\left( z\right) }$ which are in $\Gamma(\omega)$ or $\pi$.

\item Open the edges $\left( z,u\right) ,\left( z,v\right) ,\left( z,w\right) $,
  together with three disjoint self-avoiding paths $\gamma _{u},\gamma
  _{v},\gamma _{w}$ inside $\ol{B_{r+1}(z)}$ connecting $u$ to $u^{\prime }$,
  $v$ to $v^{\prime }$ and $w$ to $w^{\prime } $.
\end{enumerate}

By construction, $\omega^{\left( z\right) }\in \mathcal{X}^{\prime }$. Now given
$\omega ^{\prime }$ in the image of $\Phi$, by the same argument as in
\cite{DST}, $z$ is the only site in the new minimal path $\Gamma (\omega^{(z)})$
(from $A$ to $B$) that is connected to $C$ without using any edge in $\Gamma
(\omega^{(z)})$. Since $C \nlr[R] A$, the path $\Gamma(\omega^{(z)})$ agrees
with $\Gamma(\omega)$ up to $u^{\prime}$. Then, because $v$ is minimal among $u,
v, w, x$, $\Gamma(\omega^{(z)})$ still goes through $v^{\prime}$, and then
agrees with $\Gamma(\omega)$ from $v^{\prime}$ to the end. Therefore $\omega
^{\prime }=\omega ^{\left( z\right) }$ for some $z$ that can be uniquely
determined by $\omega^{\prime }$. Thus, the number of edges in $\{ \omega:
\text{such that } \omega ^{\prime } \in \Phi (\omega)\}$ that can vary is
bounded by the number of edges in $\overline{B_{r+2}( z) }$. This shows that
$\Phi $ satisfies the conditions of Lemma~\ref{combi0}, with
$s=\vert\overline{B_{r+2}}\vert $. This proves Fact 2 with $C_{2}=(2/\eps)
^{\vert \overline{B_{r+2}}\vert }$.

To complete the first part of the proof, we set $x =\mathbf{P}_{p}\left[
  C\lr[R] \mathcal{N}(\bar \Gamma, r)\right]$, and combine
Facts 1 and 2 to construct $h_{1}$. Notice that
\begin{equation}
\mathbf{P}_{p}\left[ \mathcal{X}\right] +\mathbf{P}_{p}\left[ \mathcal{X}%
^{\prime }\right] =x,  \label{sum}
\end{equation}%
and Fact 1 implies $\mathbf{P}_{p}\left[ \mathcal{X}\cap \left\{ \left\vert
U\right\vert \leq t\right\} \right] \leq (C_{1})^{t}\left( 1-x\right) $. Together
with Fact 2, this implies that%
\begin{equation*}
\mathbf{P}_{p}\left[ \mathcal{X}^{\prime }\right] \geq \frac{%
x-(C_{1})^{t}\left( 1-x\right) }{1+C_{2}/(t-8)}.
\end{equation*}%
Setting $t=\log \left\vert \log(1-x) \right\vert$ in the equation above, one can
easily construct $\delta>0$ and a function $h_1:[1-\delta,1]\to [0,1]$ that is
continuous, strictly increasing on $ [1-\delta,1]$, with $h_1(1-\delta)=0$ and
$h_1(1)=1$.
Set then $h_{1}\left( x\right) =0$, for $x<1-\delta$. This ends the first part
of the proof.

For the second part, we claim that for all $x\in \left[ 0,1\right] $, one can take $%
h_{2}\left( x\right) =c_{0}x$. Indeed, we can construct a map $\Phi :%
\mathcal{X}\rightarrow \mathfrak{P}(\mathcal{X}^{\prime })$ by repeating the same construction as
in the proof of Fact 2 but without needing to consider the cardinality of $U$. Since for any $\omega
^{\prime } \in \mathcal{X}^{\prime }$, the number of edges in $\{ \omega:
\omega ^{\prime } \in \Phi (\omega)\}$
that can vary is bounded, this gives $\mathbf{P}_{p}\left[ \mathcal{X}\right]
\leq C_{3}\mathbf{P}_{p}\left[ \mathcal{X}^{\prime }\right] $. Together with
\eqref{sum}, we conclude that we can take $h_{2}\left( x\right) =\left( 1+C_{3}\right)
^{-1}x$.
\end{proof}

\begin{proof}[Proof of Theorem \ref{GL0}]
  Set $\Gamma=\Gamma^S_{\rm min}(A,B)$. To see how the proof of Theorem \ref{GL}
  implies Theorem \ref{GL0}, we first note that the assumption in Theorem
  \ref{GL} that the projections of $A,B,C$ on $\mathds{Z}^{2}$ are at least
  sup-norm distance $r+1$ apart is used to deal with the issue of the
  $r$-neighborhood of $\bar \Gamma$ in \eqref{eq:19}. Indeed, as long as their
  projections on $\mathds{Z}^{2}$ are disjoint, by essentially the same proof as
 the one used for Theorem \ref{GL}, we have
\begin{equation}
    \label{eq:191}
    \Pp{C\lr[R] A}\ge \mathsf{h}_0( \Pp{C\lr[R] \bar \Gamma}).
  \end{equation}
We next note that when $R = S$, the event $\{A\lr[S] B, C\lr[S]D\}$
implies $\{C\lr[S] \bar \Gamma\}$, so by the Harris-FKG inequality,
\[ \Pp{C\lr[S] \bar \Gamma} \geq  (\Pp{A\lr[S] B}\wedge   \Pp{C\lr[S]D})^2. \]
Combining the last inequality with \eqref{eq:191} yields \eqref{eq:190}.
\end{proof}

Finally we conclude with a last gluing lemma, that will allow us to glue
together circuits. As we will see in the proof, it will be easier to glue a
circuit with a path than gluing two paths. This is due to the fact that the
local modification performed in this case does not create a new circuit, and the
reconstruction step is easier.

We now define a total ordering on circuits. The specific choice of the ordering is not important, ours is the following. A circuit is basically a path $\left( \Gamma(i)\right) _{i=1}^{r}$ in $\mathbb{S}_{k}$, such that $\Gamma(1)=\Gamma(r)$, and $\left( \Gamma (i)\right) _{i=1}^{r-1}$ is a self avoiding path.
Since we will identify circuits that differ by cyclic permutations or reverse orderings of their indices,
we will assume that the representative self avoiding path $(\Gamma(i)) _{i=1}^{r-1}$ has $\Gamma(1) \prec \Gamma(i)$
for $i >1$ and has $\Gamma(2) \prec \Gamma(r-1)$.
Given two circuits $\Gamma =\left( \Gamma (i)\right) _{i=1}^{r_{1}}$, $%
\Gamma ^{\prime }=\left( \Gamma (i)^{\prime }\right) _{i=1}^{r_{2}}$ in $%
\ol{A_{a,b}}$ that surround the origin (i.e., their
projections on $\mathbb{Z}^{2}$ have nonzero winding number around the origin), we set $\Gamma \prec
\Gamma ^{\prime }$ by using the same lexicographical ordering as we defined before for self-avoiding paths.

The following statement will be used in the renormalization argument to prove Lemma
\ref{lem:FC1}.

\begin{theorem}
  \label{cor:GL2}
  Fix $\eps>0$ and $k \geq 1$. There exists $\mathsf h_1\in\mathcal H$ such that
  for every $p \in [\eps, 1-\eps]$ and $n\ge m\ge3$,
  \begin{equation}
    \label{eq:27}
    \Pp{\Gamma_1\lr[R]\Gamma_2}\ge \mathsf h_1(f(3n,2m) a(m,n)^2),
  \end{equation}
  where $R=\ol{[0,3n]\times[-m,m]}$, $\Gamma_1$ is the minimal open circuit in
  $\ol{A_{m,n}}$ surrounding $\ol{B_m}$ {\textnormal(}$\Gamma_1=\emptyset$ if there is no such
  circuit{\textnormal)}, $\Gamma_2$ is the minimal open circuit in $\ol{A_{m,n}((3n,0))}$
  surrounding $\ol{B_m((3n,0))}$, and $a(m,n)$ denotes the probability under
  $\mathbf P_p$ that $\Gamma_1$ exists and is not empty.
\end{theorem}
\begin{proof}
  We proceed in two steps. First, we prove
\begin{equation}
    \label{eq:111}
    \Pp{\Gamma_1\lr[R]\ol{\Gamma_2}}\ge \mathsf h(\Pp{\ol{\Gamma_1}\lr[R]\ol{\Gamma_2}}),
  \end{equation}
and then
\begin{equation}
    \label{eq:112}
    \Pp{\Gamma_1\lr[R]\Gamma_2}\ge \mathsf h(\Pp{\Gamma_1\lr[R]\ol{\Gamma_2}}).
  \end{equation}
  We finish the proof by using the FKG inequality, which implies
$$ \Pp{\ol{\Gamma_1}\lr[R]\ol{\Gamma_2}} \ge f(3n,2m) a(m,n)^2.$$

Let us begin with the proof for \eqref{eq:111}. Given a configuration $\omega$,
we define $U(\omega) $ as the set of points $z\in R$ such that
  \begin{itemize}
  \item $z \in \ol{\Gamma_1(\omega)}$, and
  \item $z$ is connected to $\ol{\Gamma_2(\omega)}$ by a self-avoiding path $\gamma_z $ in $R$.
  \end{itemize}

  Let $x= f(3n,2m) a(m,n)^2$, then by the same argument as used for Fact~1 in
  the proof of Theorem \ref{GL}, we can show that there is some $C_1 <\infty$,
  such that
  \begin{equation}
    \label{eq:28}
    \Pp{\ol{\Gamma_1}\lr[R]\ol{\Gamma_2}, \Gamma_1\nlr[R] \ol{\Gamma_2}, |U|\le t} \le (C_1)^t (1-x).
  \end{equation}

  We then prove that there exists some $C_2<\infty$ such that for every $t\ge1$,
  \begin{equation}
    \label{eq:29}
    \Pp{\ol{\Gamma_1}\lr[R]\ol{\Gamma_2}, \Gamma_1\nlr[R] \ol{\Gamma_2}, |U|\ge t}\le \frac{C_2}{t} \Pp{\Gamma_1\lr[R] \ol{\Gamma_2}},
  \end{equation}

  using a map
  \begin{equation}
    \label{eq:30}
    \Phi:
  \left\vert  \begin{array}{ccc}
\{\ol{\Gamma_1}\lr[R]\ol{\Gamma_2}, \Gamma_1\nlr[R] \ol{\Gamma_2},
    |U|\ge t\}&\to& \mathfrak P \left ( \{\Gamma_1\lr[R] \ol{\Gamma_2}\} \right
    )\\
    \omega&\mapsto&\{\omega^{(z)},z\in U(\omega)\}
  \end{array}
  \right. .
  \end{equation}
  Let $\omega\in \{\ol{\Gamma_1}\lr[R]\ol{\Gamma_2}, \Gamma_1\nlr[R]
  \ol{\Gamma_2}, |U|\ge t\}$. For every $z\in U(\omega)$, the configuration
  $\omega^{(z)}\in \{\Gamma_1\lr[R] \ol{\Gamma_2}\}$ is constructed as follows.
 \begin{enumerate}
 \item Close all the edges in $\ol{B_1(z)}$ except those in $\Gamma_1(\omega)$ and
   $\gamma_z$.
 \item Let $u \in \bar z\cap \Gamma_1(\omega)$ and $v\in \bar z \cap \gamma_z$, such
   that no vertex (except $u$ and $v$) in the vertical segment between $u$ and
   $v$ belongs to $\Gamma_1(\omega)$ or $\gamma_z$. Then open all the vertical edges
   between $u$ and $v$.
\end{enumerate}

Denote by $\omega^{(z)}$ the resulting configuration. Observe that the local
modification above does not create any new circuit in $\ol{A_{m,n}}$. Otherwise,
the new circuit would contain all the vertical edges between $u$ and $v$, which
would imply some site on $\Gamma_1(\omega)$ is connected (through
$\gamma_z(\omega)$) to $\ol{\Gamma_2(\omega)}$, which contradicts
$\omega\in\{\Gamma_1\nlr[R] \ol{\Gamma_2}\}$. Therefore one can reconstruct
$\omega$ by noting that $u \in \ol{z}$ is the only site on
$\Gamma_1(\omega^{(z)})=\Gamma_1(\omega)$ that connects to
$\ol{\Gamma_2(\omega^{(z)})}=\ol{\Gamma_2(\omega)}$ without using any other
edges in $\Gamma_1$. Applying Lemma \ref{combi0} leads to \eqref{eq:29} with
$C_2 =(2/\eps)^{|\ol{B_1}|}$.

From \eqref{eq:28} and \eqref{eq:29} we can conclude \eqref{eq:111} by using the
same argument as in the proof of Theorem \ref{GL}.

Similarly, we can prove \eqref{eq:112} by defining $U(\omega)$ as the set of
points $z\in R$ such that
  \begin{itemize}
  \item $z \in \ol{\Gamma_2(\omega)}$, and
  \item $z$ is connected to $\Gamma_1(\omega)$ by a self-avoiding path $\gamma_z$ in $R$.
  \end{itemize}

And we construct a map
 \begin{equation}
   \Phi:  \left \vert
     \begin{array}{ccc}
    \{\Gamma_1 \lr[R]\ol{\Gamma_2}, \Gamma_1\nlr[R] \Gamma_2,
     |U|\ge t\}&\to& \mathfrak P \left(\{\Gamma_1\lr[R] \Gamma_2\} \right)\\
     \omega&\mapsto&\{\omega^{(z)},z\in U(\omega)\}
   \end{array}
   \right..
  \end{equation}
  Let $\omega\in \{\Gamma_1 \lr[R]\ol{\Gamma_2}, \Gamma_1\nlr[R] \Gamma_2,
  |U|\ge t\}$. For every $z\in U(\omega)$, the configuration $\omega^{(z)}$ is
  constructed as follows.
\begin{enumerate}
\item Close all the edges in $\ol{B_1(z)}$ except those in $\Gamma_2(\omega)$ and
  $\gamma_z$.
\item Let $u \in \bar z\cap \Gamma_2(\omega)$ and $v\in \bar z \cap \gamma_z$, such that
  no vertex (except $u$ and $v$) in the vertical segment between $u$ and $v$
  belongs to $\Gamma_2(\omega)$ or $\gamma_z$. Then open all the vertical edges between
  $u$ and $v$.
\end{enumerate}

As above, the local modification does not create any new circuit inside $\ol{A_{m,n}((3n,0))}$, and one can reconstruct $\omega$ from $\omega^{(z)}$ by noting that $u \in \ol{z}$
is the only site on $\Gamma_2$ that connects to $\Gamma_1$
without using any other edges in $\Gamma_2$. Applying Lemma \ref{combi0}, we obtain
\begin{equation}
    \Pp{\Gamma_1 \lr[R]\ol{\Gamma_2}, \Gamma_1\nlr[R] \Gamma_2, |U|\ge t}\le \frac{C_2}{t} \Pp{\Gamma_1\lr[R] \Gamma_2}.
  \end{equation}
The same argument as in the proof of Theorem \ref{GL} yields  \eqref{eq:112}.

\end{proof}

\subsection{Crossing estimates}

\label{sec:defin-dual-surf}

Let $R=\ol{[u,v]\times[w,t]}$ be a rectangular region in~$\mathbb{S}_k$.  Let $\mathsf{L}(R)$, $\mathsf{R}(R)$, $\mathsf{T}(R)$ and $\mathsf{B}(R)$ be
respectively the left, right, top and bottom sides of $R$.

The following proposition extends to slabs some standard estimates in planar
percolation.

\begin{proposition}
  \label{prop:standard-inequality} Let $r$ be as in Definition~\ref{def:rad1}.
  Fix $\eps>0$ and $k \geq 1$. There exists $\mathsf{h}_{2}\in \mathcal{H}$ such
  that for every $p \in [\eps, 1-\eps]$, for every $\kappa >0$, $j\geq 2$, and every
  $n\geq r+2$,
\begin{enumerate}[\bf 1.]
\item\label{item:1} $f_p(n+j\kappa n,n)\geq \mathsf h_2^{j-1} (f_p(n+\kappa n,n))$,

\item\label{item:2} $f_p(n,n+\kappa n)\geq \mathsf
  h_2^{j-1}(f_p(n,n+j\kappa n))$,

\end{enumerate}
where $\mathsf h^{j}=\underbrace{\mathsf h\circ\cdots\circ \mathsf
  h}_{j \text{ times}}$ denotes the $j$-th iterate of $\mathsf h$.
\end{proposition}

\begin{proof} 
  Let us begin with Item~\ref{item:1}. We only prove the $j=2$ case,
\begin{equation}
f_p(n+2\kappa n,n)\geq \mathsf{h}_{2}(f_p(n+\kappa n,n)).
\label{eq:9}
\end{equation}%
The more general statement (in fact a stronger result) follows by induction. For simplicity we assume that
$\kappa n$ and $n/2$ are integers. Let $R=\ol{[ 0,n+\kappa n]\times[0,n]}$, $S=\ol{[0,n]^2}$.
and $X=\ol{[0,n/2]\times \{0\}}$. Invariance under reflection and the square root
trick imply
\begin{align}
  \label{eq:37}
  \Pp{X\lr[S] \mathsf T(S)}&\ge 1-\sqrt{1-f_p(n,n)}\\
  &\ge 1-\sqrt{1-f_p(n+\kappa n,n)}.
\end{align}
Then, by the gluing lemma of Theorem~\ref{GL0}, we have
\begin{align}
  \label{eq:41}
  \Pp{X\lr[R] \mathsf R(R)}
   &\ge \mathsf h_0\big(1-\sqrt{1-f_p(n+\kappa n,n)}\big)\\
&= \mathsf h(f_p(n+\kappa n, n)),
\end{align}
where we use that $1-\sqrt{1-f} \leq f$ for $f \in [0,1]$, and define $\mathsf h(f)\dot =\mathsf{h}_0(1-\sqrt{1-f})$.

Next let $R'=\ol{[-\kappa n,n+\kappa n]\times[0,n]}$ and $Y=\ol{[n/2,n]\times \{0\}}$.
Another application of the Theorem~\ref{GL0} gluing lemma inside $R'$ gives
\begin{align}
  \label{eq:43}
  f_p(n+2\kappa n,n)&=\Pp{\mathsf L(R')\lr[R']\mathsf R(R)}\\
  &\ge\mathsf h_0 \big (\Pp{\mathsf L(R')\lr[R'] Y} \wedge  \Pp{
    X\lr[R']\mathsf R(R')}\big)\\
  &=\mathsf h_0 \big(\mathsf h (f_p(n+\kappa n, n))\big),
\end{align}
which gives exactly the statement of Eq.~\eqref{eq:9} with $\mathsf{h}_{2}= \mathsf{h}_{0} \circ \mathsf{h}$.

We now prove the second item. As we did for the first item we only prove
\begin{equation}
  \label{eq:44}
  f_p(n,n +\kappa n)\geq \mathsf{h}_{2}(f_p(n,n+2\kappa n)),
\end{equation}
and the general statement follows by induction. Consider the event that
there exists a top-down open crossing in $R'$. Then it is not hard to see that at least one of the following three
events must occur:
\begin{itemize}
\item The rectangular region $\ol{[-\kappa n,n]\times[0,n]}$ is crossed from top to bottom;
\item the rectangular region $R$ is crossed from top to bottom;
\item the square region $S$ is crossed from left to right.
\end{itemize}
The maximum probability of these three events is at least $f_p(n,n+\kappa n)$,
and the square root trick (Corollary \ref{cor:SRT}) then gives
\begin{equation}
  \label{eq:45}
  f_p(n,n+\kappa n)\ge 1-(1-\Pp{U})^{1/3} \ge 1-(1-f_p(n,n+2\kappa n))^{1/3},
\end{equation}
where $U$ denotes the union of the three events.
\end{proof}

The next theorem allows us to create an open circuit in an annulus with positive
probability. Before stating the theorem, we note that by elementary arguments
(e.g., by bounding the expected number of open self-avoiding paths of length $m$
starting from a given vertex), it is easy to see that there is some $\eps>0$
(depending only on $k$) such that $\sup_{n\geq 2}(f_\eps(2n,n-1)) <1/2$.
\begin{theorem}
  \label{cor:GL}Let $r\ge 3$ be as in Definition~\ref{def:rad1}. Fix $k\ge1$,
  and $\eps>0$ such that $\sup_{n\ge 2}(f_\eps(2n,n-1))<1/2$. For every $c>0$,
  there exists $\lambda=\lambda(c)\ge1$ and $c'>0$ such that the following
  holds. For every $p\in[\eps,1-\eps]$ and every $n\ge 4r$,
  \begin{equation}
    \label{eq:26}
    f_p(2n,n-1)\ge c \implies \Pp{\mathcal A_{\lambda n,2\lambda n}} \ge c'.
  \end{equation}
  where $\mathcal A_{\ell,2\ell}$ is the event that there exists inside
  $\ol{A_{\ell,2\ell}}$ an open circuit surrounding $\ol{B_\ell}$.
\end{theorem}

\begin{proof}
Fix $p\in[\eps,1-\eps]$, $n\ge4r$ and assume that 
\begin{equation}
  \label{eq:70}
  f_p(2n,n-1)\ge c.
\end{equation}
We may also add the restriction that  
\begin{equation}
  f_p(2n,n-1)\le 1/2.\label{eq:34}
\end{equation}
Indeed, if \eqref{eq:34} does not hold, one can lower the value of $p$ in such a
way that both \eqref{eq:70} and \eqref{eq:34} hold. The full conclusion then
follows from the monotonicity of $ \Pp{\mathcal A_{ \lambda n,2 \lambda n}}$ in~$p$.

\begin{figure}[htbp]
  \centering
  \begin{minipage}[t]{.35\textwidth}
    \centering
    \includegraphics[width=\textwidth]{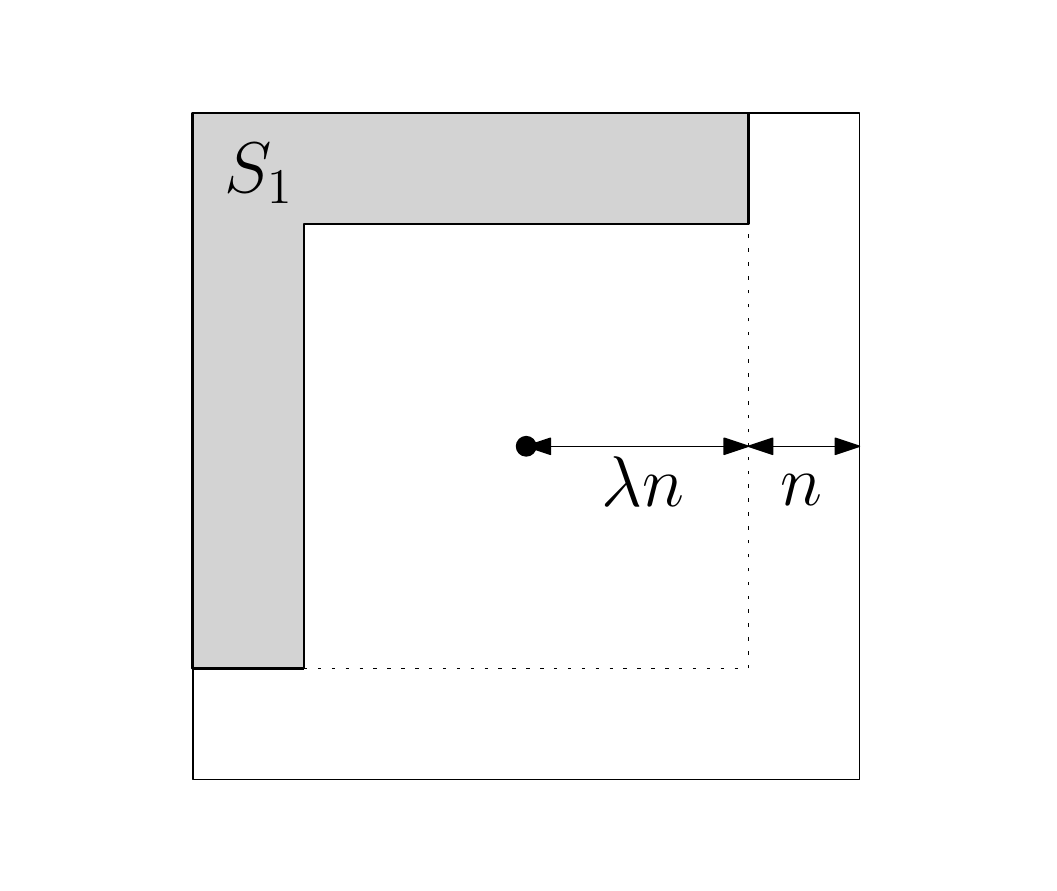}
  \end{minipage}
\begin{minipage}[t]{.35\textwidth}
    \centering
    \includegraphics[width=\textwidth]{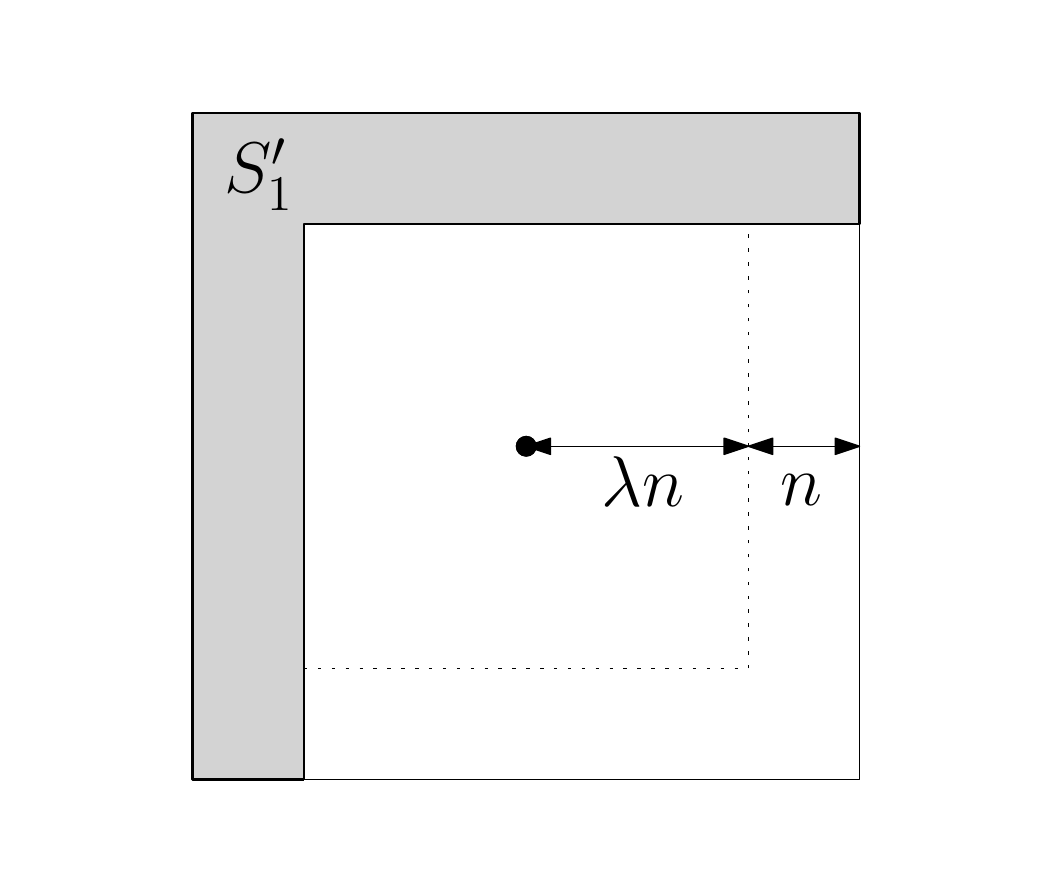}
  \end{minipage}
\\

\medskip

\begin{minipage}[t]{.35\textwidth}
    \centering
    \includegraphics[width=\textwidth]{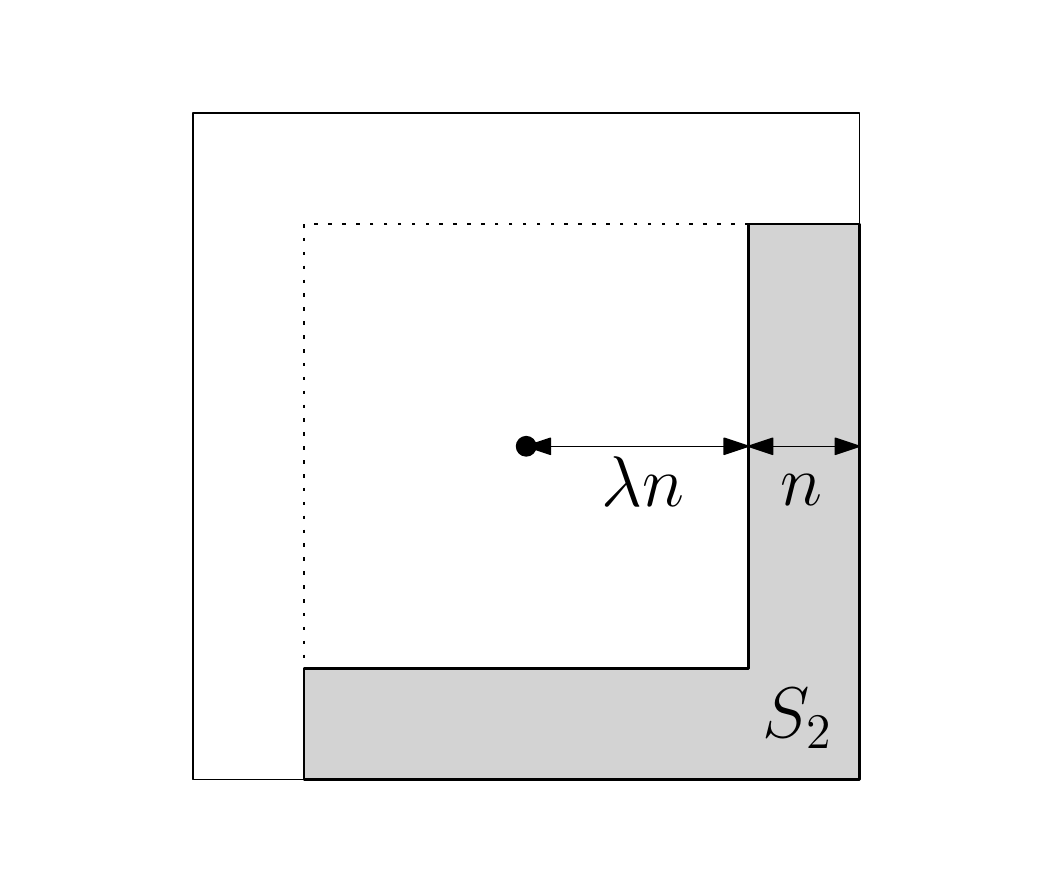}
  \end{minipage}
\begin{minipage}[t]{.35\textwidth}
    \centering
    \includegraphics[width=\textwidth]{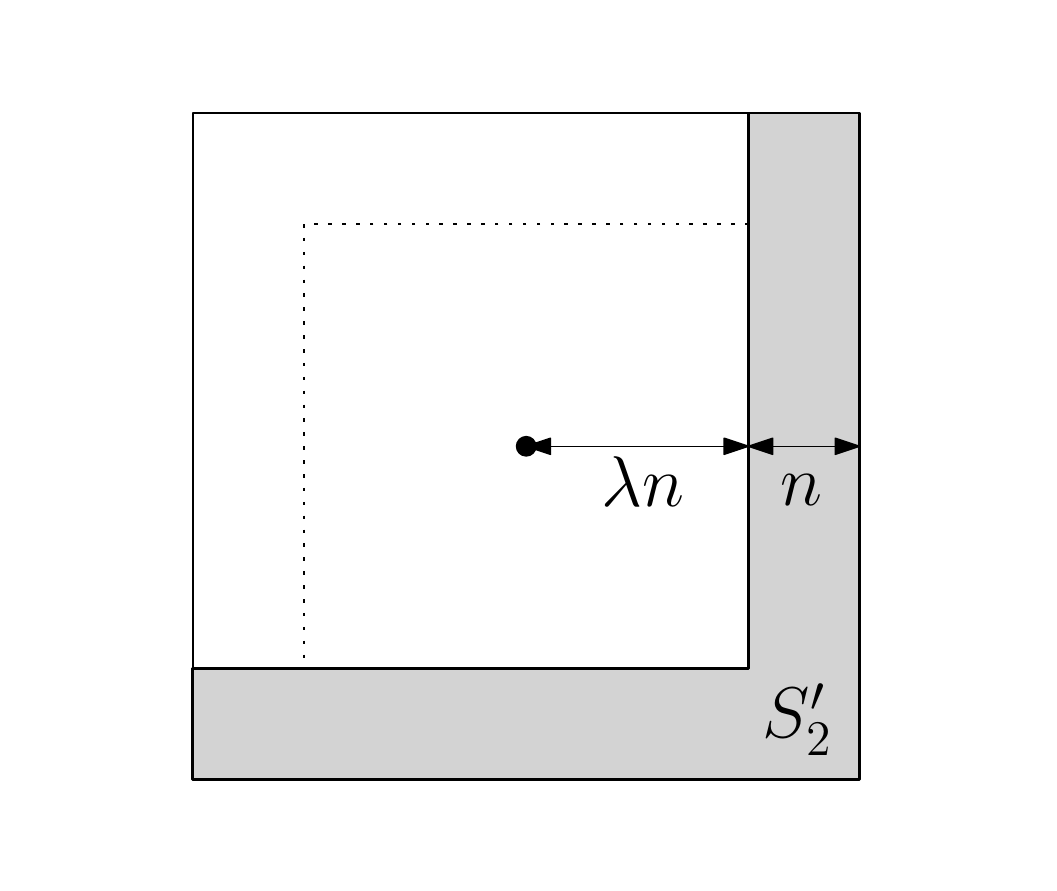}
  \end{minipage}

  \caption{Illustration of the four L-shapes used in the proof.}
  \label{fig:7}
\end{figure}

Define the following  L-shaped regions, illustrated in Fig.~\ref{fig:7}.

\begin{align}
  \label{eq:55}
   S_1=&\ol{[-\lambda n-n, \lambda n]\times (\lambda n, \lambda n+n]}
  \cup\ol{[-\lambda n-n, -\lambda n)\times [-\lambda n, \lambda n+n]},\\
  S_1'=&\ol{[-\lambda n-n, \lambda n+n]\times (\lambda n,\lambda n+n]}
  \cup\ol{[-\lambda n-n, -\lambda n)\times [-\lambda n-n,\lambda n+n]}.
 \end{align}
We also define
\begin{align}
 S_2=&\ol{[-\lambda n, \lambda n+n]\times [-\lambda n-n,-\lambda n)}
  \cup\ol{(\lambda n ,\lambda n+n]\times [- \lambda n-n, \lambda n]},\\
  S_2'=&\ol{[-\lambda n-n,\lambda n+n]\times [-\lambda n-n,-\lambda n)}
  \cup\ol{(\lambda n , \lambda n+n]\times [-\lambda n-n , \lambda n+n]}.
\end{align}
as the images of $S_1$ and $S_1'$ under the $\pi$-rotation around the origin.

Define then $\mathsf B(S_1) =\ol{[-\lambda n-n, -\lambda n)\times
  \{-\lambda n\}}$, $\mathsf B(S_1') =\ol{[-\lambda n-n, -\lambda n)\times
  \{-\lambda n-n\}}$, \\$\mathsf R(S_1) =\ol{\{\lambda n\}\times (\lambda
  n,\lambda n+n]}$, $\mathsf R(S_1') =\ol{\{\lambda n +n\}\times (\lambda
  n,\lambda n+n]}$, \\and similarly, $\mathsf L(S_2) =\ol{\{-\lambda n\}\times
  [-\lambda n-n,-\lambda n)}$, $\mathsf L(S_2') =\ol{\{-\lambda n-n\}\times
  [-\lambda n-n,-\lambda n)}$,\\ $\mathsf T(S_2) =\ol{(\lambda n , \lambda
  n+n]\times \{\lambda n\}}$, $\mathsf T(S_2') =\ol{(\lambda n , \lambda
  n+n]\times \{\lambda n +n\}}$.

The general idea in the proof is to use that with positive probability there
exists a unique cluster crossing the L-shape $S_1$, and a unique cluster
crossing $S_2$. Then we connect these two clusters at two diffferent places in
order to create an open circuit inside the annulus $\ol{A_{\lambda n,\lambda
    n+n}}$. These two connections will be obtained by two local modifications in
the top-right and bottom-left corners of the annulus $\ol{A_{\lambda n,\lambda
    n+n}}$. The uniqueness of the two clusters crossing $S_1$ and $S_2$ will be
important to avoid that the two local modifications interact. In other words,
the uniqueness requirements will prevent the second local modification from
cutting the connection created by the first local modification.

We first claim that for every $\lambda>0$ there exists a constant
$c_1=c_1(c)>0$, such that

\begin{equation}
  \label{eq:65}
  \Pp{\mathsf B(S_1') \lr[S_1'] \mathsf R(S_1')}\ge c_1 \Pp{\mathsf B(S_1)\lr[S_1] \mathsf R(S_1)}.
\end{equation}
This inequality can be obtained by performing
several gluing procedures similar to those used in the proof of
Proposition~\ref{prop:standard-inequality}, and for these gluing procedures we
use Remark 3 after Theorem~\ref{GL}.

Fix $\lambda>0$ large enough such that
\begin{equation}
2^{-\lambda}\le c_1^2/4.\label{eq:71}
\end{equation}
(The value of $\lambda$ depends on $c$ through the constant $c_1$ but does not
depend on $n$.) 

Now, we need to make precise what we mean by a unique cluster crossing $S_i$.
For $i=1,2$, let $\mathcal U_i$ be the event that there exists a unique cluster
in the configuration restricted to $S_i$ that intersects both ends of $S_i$
(i.e., the bottom and right ends of $S_1$ or respectively the left and top ends
of $S_2$). We define
\begin{align}
  &\mathcal E_0= \{\mathsf B(S_1') \lr[S_1'] \mathsf R(S_1'), \mathsf L(S_2')
  \lr[S_2'] \mathsf T(S_2')\},\text{ and}\\
  &\mathcal E= \mathcal E_0\cap\mathcal U_1\cap\mathcal U_2.
\end{align}

We wish to show that the event $\mathcal E$ occurs with probability larger than
some some positive constant. First, by the union bound we have
\begin{align}
  \Pp{\mathcal E}&\ge \Pp{\mathcal E_0} - \Pp{\mathcal E_0\setminus \mathcal
    U_1}-\Pp{\mathcal E_0\setminus \mathcal U_2}\\
&= \Pp{\mathcal E_0} -2 \Pp{\mathcal E_0\setminus \mathcal
    U_1}.\label{eq:66}
\end{align}
Then,  Eq.~\eqref{eq:65} and the Harris-FKG inequality imply
\begin{equation}
  \label{eq:68}
  \Pp{\mathcal E_0}\ge c_1^2 \Pp{\mathsf B(S_1)\lr[S_1] \mathsf R(S_1)}^2.
\end{equation}
Also, observe that the occurrence of the event $\mathcal E_0\setminus \mathcal
U_1$ implies the existence of
\begin{itemize}
\item two disjoint open paths from $\mathsf B(S_1)$ to $\mathsf R(S_1)$
  inside $S_1$,
\item an open path from $\mathsf L(S_2)$ to $\mathsf T(S_2)$ inside $S_2$.
\end{itemize}

Using first independence and then the BK inequality (see \cite{grimmett1999} for the
definition of disjoint occurrence and the BK inequality), we find 
\begin{align}
  \Pp{\mathcal E_0\setminus \mathcal U_1}&\le \Pp{\text{$S_1$ crossed by two
      disjoint open paths}} \Pp{\mathsf
    L(S_2)\lr[S_2] \mathsf T(S_2)}\\
  &\le \Pp{\mathsf B(S_1)\lr[S_1]\mathsf R(S_1)}^2 \Pp{\mathsf L(S_2)\lr[S_2]
    \mathsf T(S_2)}\\
  &\le \frac{c_1^2}4 \Pp{\mathsf B(S_1)\lr[S_1]\mathsf R(S_1)}^2.\label{eq:72}
\end{align}

For the last inequality, we use that an open path from $\mathsf L(S_2)$ to
$\mathsf T(S_2)$ inside $S_2$ must cross $\lambda$ (actually, $2 \lambda$)
disjoint $2n$ by $n-1$ rectangles in the long direction, and each of these
crossings occur with probability less than $1/2$ by Equation~\eqref{eq:34}.
Therefore, our choice of $\lambda$ in Eq.~\eqref{eq:71} gives $\Pp{\mathsf
  L(S_2)\lr[S_2] \mathsf T(S_2)}\le 2^{-\lambda}\le c_1^2/4$.

Plugging \eqref{eq:68} and \eqref{eq:72} in \eqref{eq:66}, we obtain
\begin{equation}
  \label{eq:67}
  \Pp{\mathcal E}\ge \frac{c_1^2}2\Pp{\mathsf B(S_1)\lr[S_1]\mathsf R(S_1)}^2.
\end{equation}

Finally, as in the proof of Proposition~\ref{prop:standard-inequality}, we can
use the estimate of Eq.~\eqref{eq:70} to show that $\Pp{\mathsf B(S_1)\lr[S_1]\mathsf R(S_1)}
\ge \mathsf h(c)$ for some $\mathsf h\in\mathcal H$ (that depends on $\lambda$).
Therefore, there exists a constant $c_2=c_2(c,\lambda)>0$ such that
\begin{equation}
  \label{eq:69}
  \Pp{\mathcal E}\ge c_2.
\end{equation}

We claim that there exists a constant $c_3>0$ such that $\Pp{\mathcal A_{\lambda
    n,\lambda n+n}} \ge c_3^2 \Pp{\mathcal E}$, which will then finish the
proof. To show this, we now perform a two step gluing procedure in the square
regions $R_1\dot =\ol{(\lambda n,\lambda n +n]^2}$ and $R_2\dot =\ol{[-\lambda n
  -n ,-\lambda n)^2}$ to create an open circuit. When $\mathcal E$ occurs, we
denote by $\Gamma_1$ the minimal open self-avoiding path  from
$\mathsf B(S_1)$ to $\mathsf R(S_1')$ inside \[\ol{[-\lambda n-n, \lambda n+n]\times (\lambda n,\lambda n+n]}
\cup\ol{[-\lambda n-n, -\lambda n)\times [-\lambda n,\lambda n+n]}.\]
We first prove
\begin{equation}
    \label{eq:121}
    \Pp{\Gamma_1\lr[S_2']\mathsf L(S_2'), \mathcal E}\ge c_3 \Pp{\mathcal E},
  \end{equation}
and then
\begin{equation}
    \label{eq:122}
   \Pp{\mathcal A_{\lambda n,\lambda n+n}} \ge c_3 \Pp{\Gamma_1\lr[S_2']\mathsf L(S_2'), \mathcal E}.
  \end{equation}

  Let us begin with the proof for \eqref{eq:121}. Similarly to the proof of
  Theorem \ref{GL}, we construct a map
  \begin{equation}
    \Phi:
    \left\vert  \begin{array}{ccc}
        \{\Gamma_1\nlr[S_2']\mathsf L(S_2'),  \mathcal E\}&\to&  \{\Gamma_1\lr[S_2']\mathsf L(S_2'), \mathcal E\} \\
        \omega&\mapsto& \omega^{(z)}
      \end{array}
  \right. ,
  \end{equation}
  where the configuration $\omega^{(z)}$ is defined as follows. Given $\omega
  \in \{\Gamma_1\nlr[S_2']\mathsf L(S_2'), \mathcal E\}$, we first choose the
  minimal point $z\in R_1$ such that
  \begin{itemize}
  \item $z \in \ol{\Gamma_1(\omega)}$, and
  \item $z$ is connected to $\mathsf L(S_2')$ by an open self-avoiding path in
    $S_2'$.
  \end{itemize}
  Then we construct the  configuration $\omega^{(z)}$ by the following three
  steps:

\begin{enumerate}
\item Define $u^{\prime },v^{\prime }$ to be respectively the first and last
  vertices (when going from $\mathsf B(S_1)$ to $\mathsf R(S_1')$) of
  $\Gamma_1(\omega)$ which are in $\overline{B_r\left( z\right) }\cap R_1$.
  Choose $w^{\prime }$ on the boundary of $%
  \overline{B_r\left( z\right) }\cap R_1$, such that there exists an open
  self-avoiding path $\pi$ from $w^{\prime }$ to $\mathsf L(S_2')$ inside
  $S_2'$. The points $u'$, $v'$ and $w'$ are all distinct, and by
  Definition~\ref{def:rad1}, we can choose $u,v,w$ such that $\left( z,u\right)
  ,\left( z,v\right)$ and $\left( z,w\right) $ are three distinct edges with $%
  v \prec w $.

\item Close all the edges of $\omega $ in $\overline{B_{r+1}\left( z\right)
  }\cap \ol{ [\lambda n,\lambda n+n]^2}$ except the edges of
  $\overline{B_{r+1}\left( z\right) }\backslash \overline{B_{r}\left( z\right)
  }\cap \ol{[\lambda n,\lambda n+n]^2}$ which are in $\Gamma_1$ and $\pi$.

\item Open the edges $\left( z,u\right) ,\left( z,v\right) ,\left( z,w\right) $,
  together with three disjoint self-avoiding paths $\gamma _{u},\gamma
  _{v},\gamma _{w}$ inside $\ol{B_{r}(z)}\cap R_1$ connecting $u$ to $u^{\prime }$,
  $v$ to $v^{\prime }$ and $w$ to $w^{\prime } $.
\end{enumerate}

By construction, $\omega^{(z)} \in \{\Gamma_1\lr[S_2']\mathsf L(S_2'), \mathcal
E\}$. The uniqueness of the cluster crossing $S_1$ implies that the path
$\Gamma_1(\omega^{(z)})$ agrees with $\Gamma_1(\omega)$ from $\mathsf B(S_1)$ to
$u'$, and we can reconstruct the point $z$ by noting that $z$ is the only site
in $\Gamma_1 (\omega^{(z)})$ that is connected to $\mathsf L(S_2')$  without using any edge in $\Gamma_1 (\omega^{(z)})$. Therefore $\Phi $
satisfies the conditions of Lemma~\ref{combi0}, with
$s=\vert\overline{B_{r+1}}\vert $. Applying Lemma~\ref{combi0} then yields
\eqref{eq:121}.

We next move to the proof of \eqref{eq:122}. We define $\Gamma_2$ as the minimal
open crossing from $\mathsf T(S_2)$ to $\mathsf L(S_2')$ in \[ \ol{[-\lambda n-n, \lambda n+n]\times [-\lambda n -n,-\lambda n)}
\cup\ol{(\lambda n , \lambda n+n]\times [-\lambda n-n, \lambda n]}.\]

As before, we construct a map
\begin{equation}
      \Phi:
  \left\vert  \begin{array}{ccc}
\{\mathcal A_{\lambda n,\lambda n+n}^{c}, \Gamma_1\lr[S_2']\mathsf L(S_2'),  \mathcal E\}&\to&  \mathcal A_{\lambda n,\lambda n+n}\\
    \omega&\mapsto&\omega^{(z)}
  \end{array}
  \right. .
  \end{equation}

  Let $\omega \in \{ \mathcal A_{\lambda n,\lambda n+n}^{c},
  \Gamma_1\lr[S_2']\mathsf L(S_2'), \mathcal E\}$. We choose a 
point $z \in \ol{\Gamma_2(\omega)}$ which is connected to $\mathsf R(S_1)$ 
inside $S_1'$. We construct the
  configuration $\omega^{(z)}$ by essentially the same three steps as we did in
  proving \eqref{eq:121}, the only difference is now we do local modifications
  in $\ol{B_{r+1}(z)} \cap \ol{[-\lambda n -n,-\lambda n]^2}$, and
  $u^\prime,v^\prime$ are defined respectively to be the first and last vertices
  (when going from from $\mathsf T(S_2)$ to $\mathsf L(S_2')$) of
  $\Gamma_2(\omega)$ which are in $\overline{B_{r}\left( z\right) }\cap R_2$, and $w'$ is connected by a self-avoiding path to $\mathsf 
R(S_1)$ inside $S_1'$.

  To see that $\omega^{(z)} \in \mathcal A_{\lambda n,\lambda n+n}$, we first
  note that the local modification above does not change the 'unique clusters'
  inside $S_1$ and $S_2$ (they are measurable with respect to the edge variables
  in $S_1$ and $S_2$). Then, since $\omega \in \mathcal A_{\lambda n,\lambda n+n
}^{c}$, the path $\Gamma_2(\omega)$ cannot be connected to $\mathsf
  R(S_1)$. Otherwise, the fact that $ \Gamma_1\lr[S_2']\mathsf L(S_2')$ and the
  uniqueness of the cluster in $S_1$ would imply the existence of a circuit in
  $\ol{A_{\lambda n,\lambda n+n}}$. Therefore $\Gamma_2(\omega^{(z)})$ agrees
  with $\Gamma_2(\omega)$ up to $u'$, and it must go through $z$ and exit
  $\ol{B_{r}(z)}$ through $v'$. Then it agrees with $\Gamma_2(\omega)$ from
  $v'$ to the end. Thus $ \omega^{(z)} \in \mathcal A_{\lambda n,\lambda n +n}$, and one can reconstruct the point $z$ by noting that $z$ is the only
  site in $\Gamma_2 (\omega^{(z)})$ that is connected to $\mathsf R(S_1)$ without using
  any edge in $\Gamma_2 (\omega^{(z)})$. Therefore $\Phi $ satisfies the
  conditions of Lemma~\ref{combi0}, with $s=\vert\overline{B_{r+1}}\vert $.
  Applying Lemma~\ref{combi0} then yields \eqref{eq:122}, and thus concludes the
  proof.

\end{proof}

\subsection{Renormalization inputs}

\label{sec:renorm-inputs}

\begin{lemma}[Finite criterion for $\protect\theta (p)>0$]
  \label{lem:FC1} Fix $k\ge 0$ and $\eps>0$ such that $\eps<p_c(\mathds
  S_k)<1-\eps$. There exists a constant $c_{1}>0$, such that the following
  holds. For every $p\in [\eps, 1-\eps]$  and every $n \ge 4r$,
  \begin{equation}
    f_p\left( 2n,n-1\right) >1-c_1 \implies \mathbf{P}_p\left[ 0\overset{\mathbb{S}_{k}}{\longleftrightarrow }\infty \right] >0.\label{eq:73}  
  \end{equation}
\end{lemma}

\begin{proof} 
  Fix $k\ge 0$ and $\eps>0$ as in the statement of the lemma. We first prove the
  following claim, which isolates the renormalization argument we are using.
  \begin{claim}
    There exist $\eta>0$ such that for every  $p\in [\eps, 1-\eps]$  and every
    $n\ge m \ge 4r$,
    \begin{equation}
      \label{eq:74}      
      \left.
        \begin{aligned}
          f_p(3n,2m)\ge 1-\eta\\ 
          \Pp{\mathcal A_{m,n}}\ge1-\eta 
        \end{aligned}\   \right\}\implies  \mathbf{P}_p\left[ 0\overset{\mathbb{S}_{k}}{\longleftrightarrow }\infty \right] >0.
    \end{equation}
  \end{claim}
  \begin{proof}[Proof of Claim]
    Fix $p_{0}<1 $ to be such that any $1$-dependent bond percolation measure on
    $\mathbb{Z}^{2}$ with marginals larger than $p_{0}$ produces an infinite
    cluster. (This is well defined by standard stochastic domination
    arguments~\cite%
    {ligget1997domination} or a Peierls argument~\cite{balister2005continuum}).
    Let $G=(V,E)$ be the graph with vertex set $V:=3n\Z^2$, and edge set
    $E:=\{\{v,w\}:|v-w|=3n\}$. It is a rescaled version of the standard
    two-dimensional grid $\mathbb{Z}^{2}$.

    We define a percolation process $X$ on $G$ follows. Consider a Bernoulli
    percolation process with density $p$ on the slab $\mathbb S_{k}$. Let
    $e=\{u,v\}\in E$ be a horizontal edge with  $v=u+(3n,0)$. Set $X(e)=1$ if

\begin{itemize}
\item There exists an open circuit in $\ol{A_{m,n}(u)}$ surrounding
  $\ol{B_m(u)}$, and an open circuit in $\ol{A_{m,n}(v)}$ surrounding
  $\ol{B_m(v)}$.
\item the minimal open self-avoiding circuit inside $\ol{A_{m,n}(u)}$ is
  connected to the minimal open self-avoiding circuit inside $\ol{A_{m,n}(v)}$
  by an open path that lies inside $\ol{u+ ([0,3n]\times [-m,m])}$.
\end{itemize}
Set $X(e)=0$ otherwise. Define $X(e)$ analogously when $e$ is a vertical edge.

By Theorem~\ref{cor:GL2}, the condition on the left hand side of \eqref{eq:74}
 implies
\begin{equation}
\mathbf{P}_{p}[ X(e)=1] \ge \mathsf h_1(f_p(3n, 2m)\Pp{\mathcal A_{m,n}}^2)\geq  \mathsf h_1((1-\eta)^3). \label{eq:3}
\end{equation}%
If we choose $\eta$ small enough, the above probability is larger than $p_0$.
Since the percolation process $X$ is $1$-dependent, there exists with positive
probability an infinite self avoiding path in $G$ made of edges $e$ satisfying
$X(e)=1$. This implies that in the slab, we have \mbox{$\Pp%
  {0\overset{\mathbb{S}_{k}}{\longleftrightarrow }\infty}>0$}. This ends the
proof of the claim.
\end{proof}
We now prove the lemma. By Theorem~\ref{cor:GL} one can first choose two
constants $\lambda>0$
and $c'>0$ such that for every $p\in[\eps,1-\eps]$ and every $n\ge 4r$ 
\begin{equation}
  \label{eq:75}
    f_p(2n,n-1)\ge 1/4 \implies \Pp{\mathcal A_{\lambda n,2\lambda n}} \ge c'.
\end{equation}
Then we choose a constant $\ell <\infty$ such that $(1-c')^\ell <\eta$. By
Item~1 of Proposition \ref{prop:standard-inequality}, we can finally choose
$c_1>0$ small enough such that
\begin{equation}
  \label{eq:76}
  f_p(2m,m-1)>1-c_1 \implies f_p(2^\ell\lambda m,m-1)\ge 1-\eta\ge 1/4.
\end{equation}
One can easily check that this choice of $c_1>0$ concludes the proof. Assume
that for some $m\ge 4r$, $f_p(2m,m-1)>1-c_1$. Eq.~\eqref{eq:76} and \eqref{eq:75}
give for every $0\le i\le \ell-1$, $\Pp{\mathcal A_{2^i\lambda m,2^{i+1}\lambda
    m}} \ge c'$.  Therefore, by independence we have
  \begin{equation}
    \label{eq:58}
    \Pp{\mathcal A_{m,2^\ell \lambda m}} \ge 1-(1-c')^\ell\ge  1-\eta.
  \end{equation}
  The claim above applied to $n=2^\ell \lambda m$ implies that $\mathbf{P}_p\left[ 0\lr[\mathbb{S}_{k}]\infty \right] >0$.

\end{proof}

\begin{lemma}[Finite criterion for exponential decay]
\label{lem:FC2} There exists an absolute constant $c_{2}>0$, such that $%
f_p(n,2n)<c_{2}$ for some $n\geq 1$, implies that for every $m$, $\mathbf{P}_{p}%
\left[ 0\overset{\mathbb{S}_{k}}{\longleftrightarrow }\partial \ol{B _{m}}%
\right] \leq e^{-cm}$. 
\end{lemma}

\begin{proof}
  This result is standard and can be proved in various ways. We present here a
  renormalization argument of Kesten \cite{kesten1982percolation}. Let $n\ge1$.
  If the rectangle $[0,2n]\times[0,4n]$ is crossed horizontally, then both
  rectangles $[0,n]\times[0,4n]$ and $[n,2n]\times[0,4n]$ must be crossed
  horizontally. Using translation invariance and independence, we obtain
  \begin{equation}
    \label{eq:23}
    f_p(2n,4n)\le f_p(n,4n)^2.
  \end{equation}
  Now consider the covering of $[0,n]\times[0,4n]$ by the five $n$ times $2n$
  rectangles $R_1,\ldots,R_5$ illustrated on Fig.~\ref{fig:covering}.
  \begin{figure}
    \centering
    \hfill
    \begin{minipage}[t]{.2\linewidth}
      \includegraphics[width=.98\linewidth]{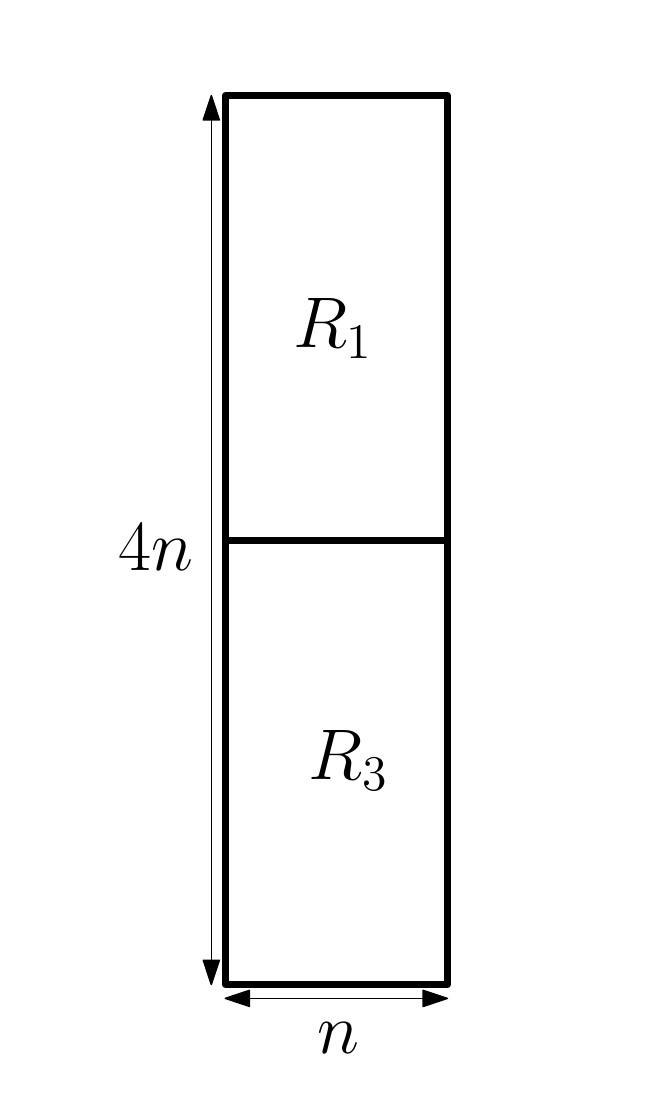}
    \end{minipage}
    \hfill
    \begin{minipage}[t]{.2\linewidth}
      \includegraphics[width=.98\linewidth]{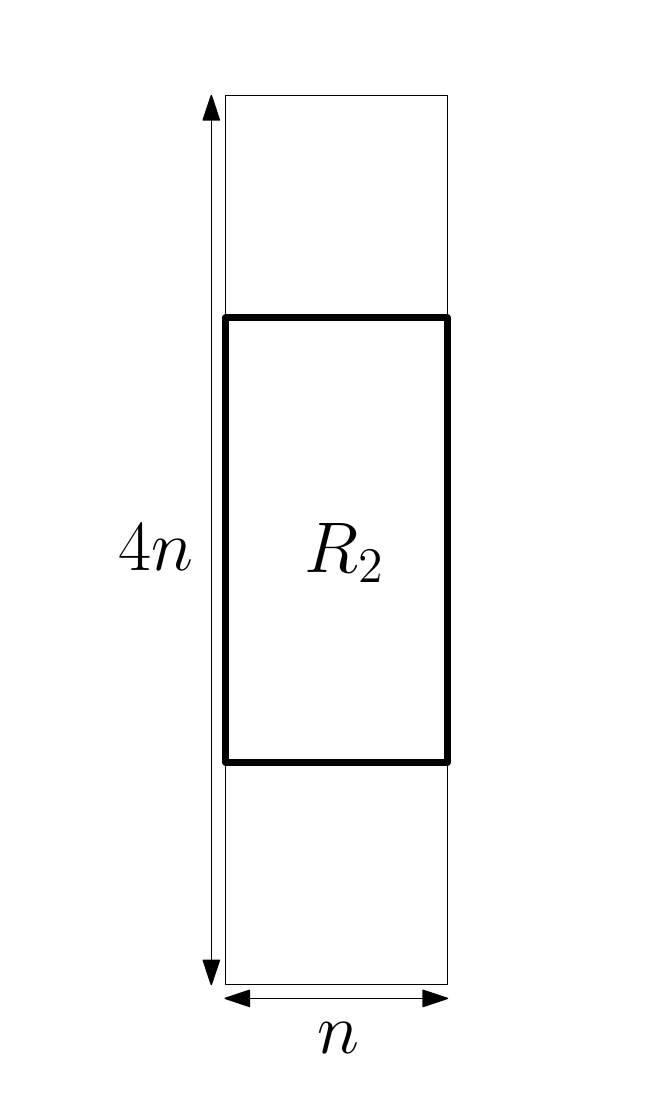}
    \end{minipage}
    \hfill
    \begin{minipage}[t]{.2\linewidth}
      \includegraphics[width=.98\linewidth]{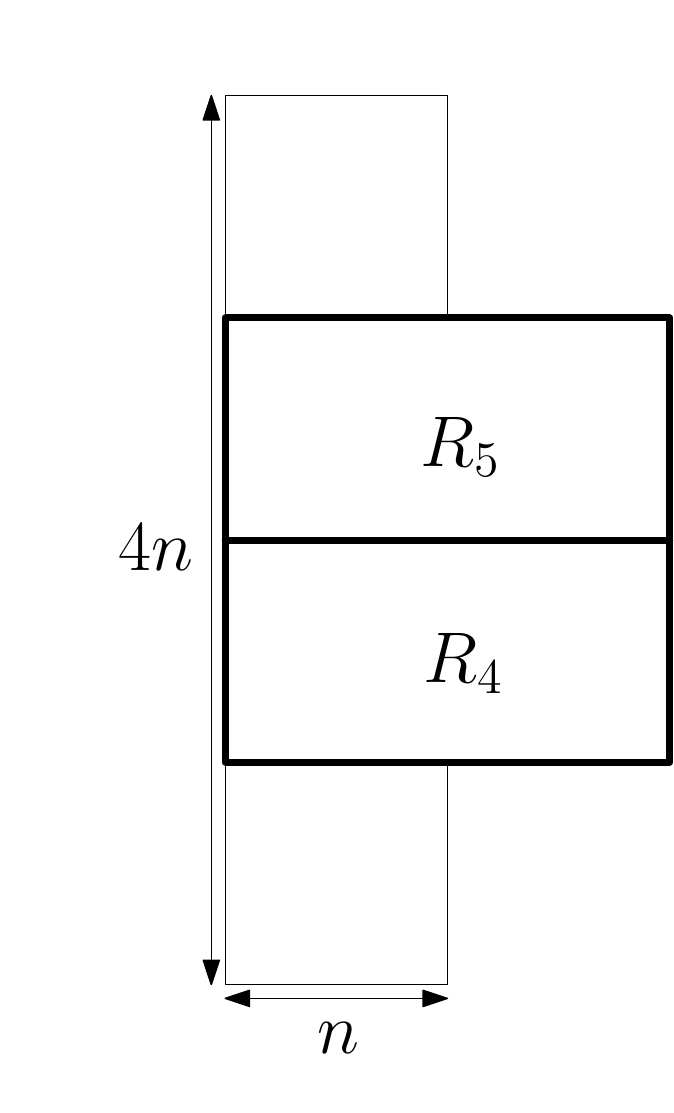}
    \end{minipage}
    \hfill
    \caption{A covering of $[0,n]\times[0,4n]$ by five $n$ times $2n$ rectangles.}\label{fig:covering}
  \end{figure}
  If $[0,n]\times[0,4n]$ is crossed horizontally then at least one of the five
  rectangles $R_1,\ldots, R_5$ must be crossed in the easy direction. Using
  translation invariance and the union bound, we find
  \begin{equation}
    \label{eq:24}
    f_p(n,4n)\le 5 f_p(n,2n).
  \end{equation}
  Together with Eq.~\eqref{eq:23} we obtain for every $n\ge1$,
  \begin{equation}
    \label{eq:25}
    f_p(2n,4n)\le 25 f_p(n,2n)^2.
  \end{equation}
  If $f_p(n_0,2n_0)<1$ for some $n_0\ge1$, Eq.~\eqref{eq:25} implies by
  induction that the sequence $f_p(n,2n)$ decays exponentially fast in $n$,
  which easily implies that the probability for $0$ to be connected to $\partial
  \ol{B _{n}}$ decays exponentially.
\end{proof}

\begin{lemma}
\label{lem:renormalization-inputs} For critical Bernoulli percolation on
$\mathbb S_{k}$ (i.e., $p=p_{c}(\mathbb S_{k})$) we have $f_p(2n,n-1)\leq 1-c_{1}$ and $%
f_p(n,2n)\geq c_{2}$ for every $n \ge 4r$.
\end{lemma}

\begin{proof}
Take $\eps >0$ such that $p_c \in (\eps, 1-\eps)$. Consider the set \\$\{p\in (\eps, 1-\eps):\text{there exists $n \ge 4r$ s.t. }%
f_p(2n,n-1)>1-c_{1}\}$. It is open and does not intersect $[0,p_{c}(\mathbb %
S_{k}))$ (by Lemma~\ref{lem:FC1}). Thus,
$p_{c}$ does not belong to this set, and therefore $f_p(2n,n-1)\leq 1-c_{1}$ for
every $n \ge 4r$ and $p\leq p_{c}$. Similarly, the inequality $%
f_p(n,2n)\geq c_{2}$ at $p=p_{c}$ follows from Lemma~\ref{lem:FC2}.
\end{proof}

\subsection{The RSW-Theorem: positive probability version\label{sec:RSWpf}}

\begin{theorem}
\label{thm:RSW} Fix $\eps>0$ and $k\ge1$. For $p\in \lbrack \eps,1-\eps]$, the following implication holds
for the horizontal crossing probability $f(m,n)=f_p(m,n)$.
\begin{equation}
\text{If $\displaystyle\inf_{n\ge1}{f(n,2n)}>0$, then
$\displaystyle\inf_{n\ge1}{f(2n,n)}>0$.}\label{eq:48}
\end{equation}

\end{theorem}

All this section is devoted to the proof of this theorem. Under the assumption of the theorem, we can choose a
constant $c_0>0$ such that for every $n\ge1$,
\begin{equation}
  \label{eq:38}
  f(n,2n)\ge c_0.
\end{equation}
We will use several gluing lemmas presented in Section~\ref{sec:gluing-lemma}.
To this end, we fix a number $r\ge 3$ as in Definition~\ref{def:rad1}. In the
proof below we use various constants denoted $c_j$, each is independent of $n$.

Define $S=\ol{[0,7n]\times[0,8n]}$, $R=\ol{[-7n,7n]\times[0,13n]}$, $
X=\ol{\{7n\}\times[0,4n]}$ and $Y=\ol{\{7n\}\times[5n,13n]}$ (see
Fig.~\ref{fig:1} for an illustration). Let $\mathcal A$ be the event that there exist
\begin{itemize}
\item an open path in $S$ from its left side to $X$ and
\item an open path in $R$ from its bottom side to $Y$.
\end{itemize}

\begin{figure}[htbp]
  \centering
  \includegraphics[width=.43\linewidth]{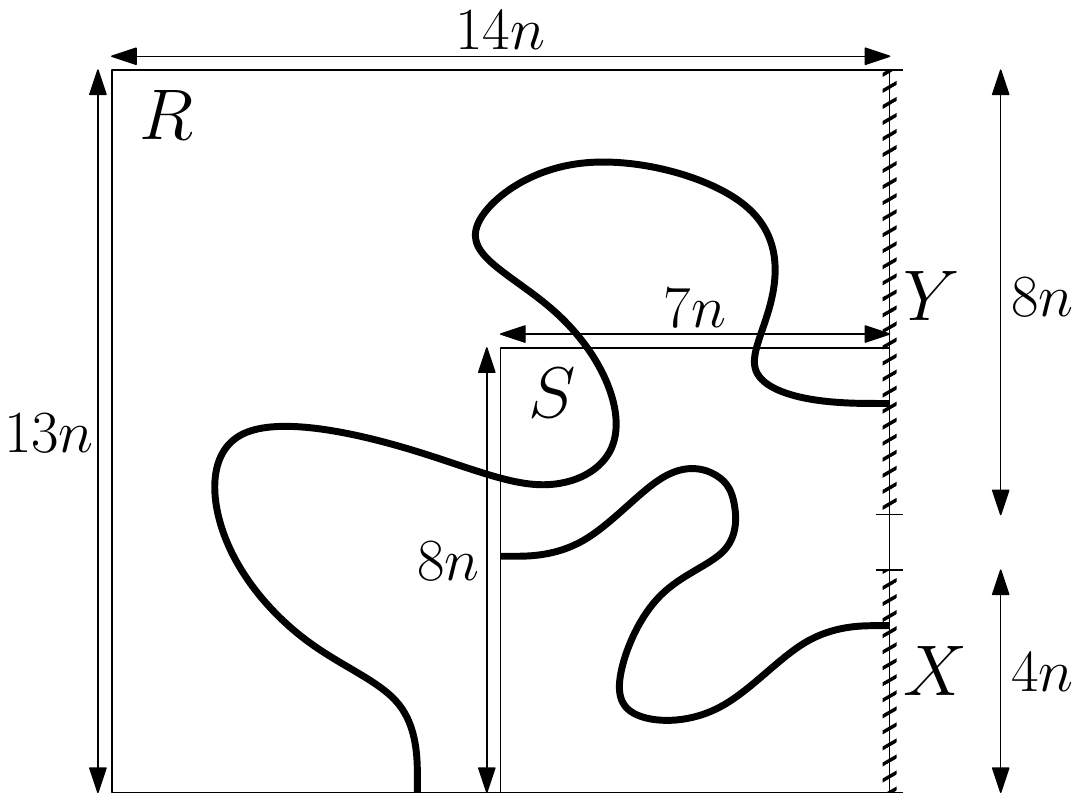}
  \caption{Diagrammatic representation of the event $\mathcal
    A$.}
  \label{fig:1}
\end{figure}

\begin{lemma}\label{lem:RSW1}
  Assume that Eq.~\eqref{eq:38} holds. Then there exists a constant $c_4>0$
  independent of $n\ge 4r$ such that
  \begin{equation}
    \label{eq:39}
    \Pp{\mathcal A}\ge c_4.
  \end{equation}
\end{lemma}
\begin{proof}
  We will prove that there exist constants $c_1,c_3>0$ such that
  \begin{align}
    \label{eq:7}&\Pp{\mathsf L(S)\lr[S] X}\ge c_1\text{, and}\\
    \label{eq:12}&\Pp{ \mathsf B(R)\lr[R] Y}\ge c_3.
  \end{align}
  Lemma~\ref{lem:RSW1} then follows by the Harris-FKG inequality with $c_4 = c_1c_3$.

  By Eq.~\eqref{eq:38}, we have $f(7n,14n)\ge c_0$. Therefore, by Item 2 of
  Proposition~\ref{prop:standard-inequality} (with $\kappa=1/7$ and $j=7$), we have $f(7n,8n)\ge \mathsf
  h_2^6(c_0)>0$. In other words, the rectangle $S$ is crossed horizontally by
  an open path with probability larger than $\mathsf h_2^6(c_0)$. Using a
  symmetry and the union bound, we obtain Eq.~\eqref{eq:7} with
  $c_1=\mathsf h_2^6(c_0)/2$.

  Let us now prove Eq.~\eqref{eq:12}. Since  $f(14n,28n)$ and
  $f(13n,26n)$ are at least $c_0$,
  Item 2 of Proposition~\ref{prop:standard-inequality} implies
  \begin{align}
    \label{eq:31}&f(14n,16n)\ge \mathsf h^{6}_2(c_0) \ge c_2,\text{ and}\\
    \label{eq:32}&f(13n,14n)\ge \mathsf h^{12}_2(c_0) \ge c_2,
  \end{align}
  where $c_2=\min\{\mathsf h^{6}_2(c_0), \mathsf h^{12}_2(c_0)\}$. Consider the event that there exists inside \\
  $K=\ol{[-7n,7n]\times[-3n,13n]}$ an open path $\Pi$ from $\mathsf L(K)$ to
  $Y$. Note that $Y$ is the top half of $\mathsf R(K)$. By \eqref{eq:31} and a symmetry, this occurs with probability larger
  than $c_2/2$. When the path $\Pi$ exists, either it touches the bottom side of
  $R$, or it remains inside $R$. Hence, by the union bound, at least one of the
  following two cases holds:
  \begin{itemize}
    \item $\Pp{ \mathsf B(R)\lr[R] Y} \ge c_2/4$;
    \item $\Pp{ \mathsf L\,(R)\lr[R] Y} \ge c_2/4$.
  \end{itemize}
  The first case gives exactly \eqref{eq:12}. In the second case we can conclude the proof by
  using the Theorem~\ref{GL0} gluing lemma inside $R$. We would know that $R$
  is crossed from left to right by an open path with probability larger than
  $c_2/4$, and from top to bottom with probability larger than $c_2$ (by
  Eq.~\eqref{eq:32}). Theorem~\ref{GL0} would then imply Eq.~\eqref{eq:12} with $c_3=\mathsf h_0(c_2/4)$.
\end{proof}

We now investigate a possible geometry of connecting paths when $\mathcal{A}$
occurs, which will be used near the end of the proof of Theorem~\ref{thm:RSW}.
Let $\gamma$ be a deterministic path from $X$ to $\mathsf L(S)$ in $S$ such that
$\gamma\cap Y=\emptyset$. Write $\gamma'$ for the symmetric reflection of
$\gamma$ through the plane $\{0\}\times\R^2$. Notice that the set
$\overline{\gamma \cup \gamma'}$ disconnects the top side $\mathsf T(R)$ of $R$
from its bottom side $\mathsf B(R)$, in the sense that any path from top to
bottom in $R$ must intersect $\overline{\gamma \cup \gamma'}$. Let $K_0(\gamma)$
be the connected component of $\mathsf T(R)$ in $R\setminus\overline{\gamma \cup
  \gamma'}$. Then, set $K(\gamma)=(K_0(\gamma) \cup \partial
K_0(\gamma))\setminus \mathcal N(\ol\gamma,3r)$, where we note that every edge in
the boundary, $\partial K_0(\gamma)$, of $K_0(\gamma)$ is between a vertex in
$K_0(\gamma)$ and one in $\overline{\gamma \cup \gamma'}$. We recall that $r$ is
given by Definition~\ref{def:rad1}, and $\mathcal N(\ol\gamma,3r)$ is the set of
vertices within sup-norm distance $3r$ of $\ol\gamma$. Define $\mathcal C_\gamma$
as the event that there exists an open path in $K(\gamma)$ from $Y$ to
$\overline{\gamma'}$. (See Fig.~\ref{fig:2}.)
\begin{figure}[htbp]
  \centering
  \includegraphics[width=.43\linewidth]{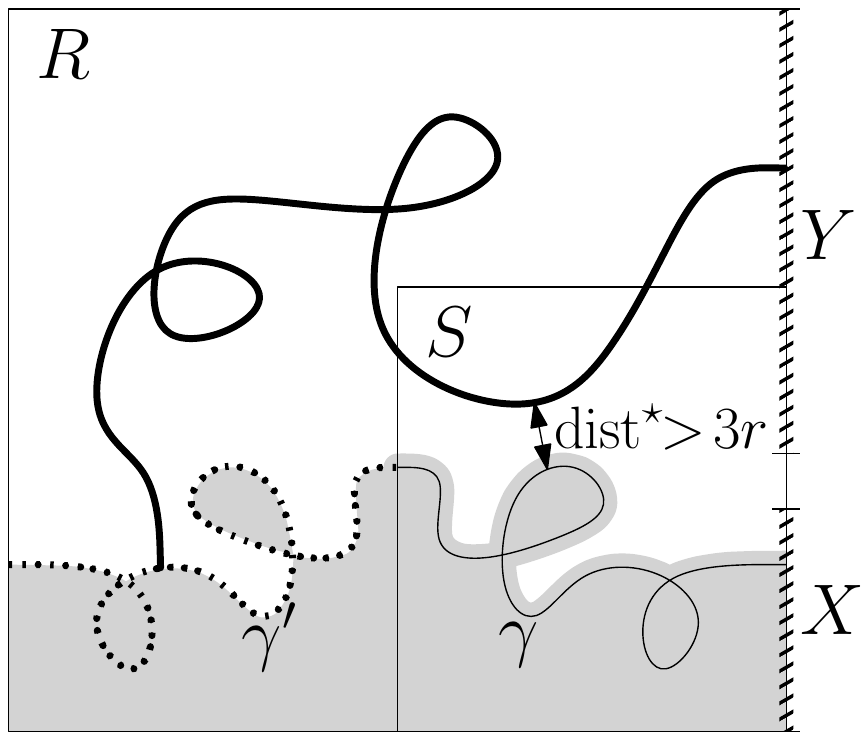}
  \caption{A diagrammatic representation of the event $\mathcal C_\gamma$ that there
    is an open path from $Y$ to $\ol{\gamma'}$ inside the region $K(\gamma)$ (which
    corresponds to the complement in $R$ of the grey region).}
\label{fig:2}
\end{figure}
\begin{lemma}\label{lem:RSW3}
  There exists  $c_7>0$ such that for every $n\ge6r$,
  \begin{equation}
    \label{eq:35}
   \max_\gamma (\P{\mathcal C_{\gamma}})\ge c_4/3\implies f(14n,13n)\ge  c_7.
  \end{equation}
  where the maximum is taken over all deterministic paths $\gamma$ from $X$ to
  $\mathsf L(S)$ in $S$, such that $\gamma\cap Y=\emptyset$.
\end{lemma}

\begin{proof}
  Take $\gamma$ such that $\P{\mathcal C_{\gamma}} \ge c_4/3$. Let $Y'$ denote
  the symmetric reflection of $Y$ through the plane $\{0\}\times\R^2$. Assume
  for simplicity that $n$ is a multiple of $2r$.  Define 
  \[K_\Box\dot{=}\bigcup_{z\in 2r\Z^2\text{ s.t. } \ol{B_{r}(z)}\subset
    K(\gamma)} \ol{B_r(z)}.\] We consider the left- and right-bottom parts of
  the boundary of $K_\Box$, defined respectively by $A=\partial K_\Box\cap
  \ol{(-7n,0)\times[0,13n)}$ and $A'=\partial K_\Box\cap
  \ol{(0,7n)\times[0,13n)}$. Observe that
     \begin{equation}
       \label{eq:77}
       \Pp{A\lr[K_\Box]Y}\ge \P{\mathcal C_{\gamma}} \ge c_4/3.
     \end{equation}
The domain $K_\Box$ is regular enough to apply Theorem~\ref{GL0} gluing Lemma
(see Remark \ref{item:6} after Theorem~\ref{GL}). We obtain 
 $$\Pp{Y\lr[K_\Box] Y'}\ge\mathsf{h}_0 (\Pp{A\lr[K_\Box]Y}\wedge\Pp{A'\lr[K_\Box]Y'}),$$
for some $\mathsf{h}_0 \in \mathcal{H}$. 
This gives 
\begin{equation}
  \label{eq:78}
  f(14n,13n)\ge c_7,
\end{equation}
with $c_7 = \mathsf{h}_0 (c_4/3)$, which concludes the proof.

\end{proof}

\begin{proof}[Proof of Theorem~\ref{thm:RSW}]
  Assume that Eq.~\eqref{eq:38} holds. By Lemma~\ref{lem:RSW1} we have
  $\Pp{\mathcal A}\ge c_4$.

  Let $\Gamma=\Gamma^S_{\mathrm{min}}(X,\mathsf L(S))$ be the minimal open path
  from $X$ to $\mathsf L(S)$ inside $S$. (Recall that $\mathcal{A}$ does not occur if
  there is no such path.) Let $\mathcal B_1$ be the event that there exists an
  open path from $Y$ to $X$ inside $R$ and $\mathcal B_2$ be the event that
  there exists an open path from $Y$ to $\mathcal N(\ol \Gamma,3r)$ inside $R$ (see Fig.~\ref{fig:3}).
  \begin{figure}[htbp]
  \hfill
  \begin{minipage}{.43\linewidth}
    \centering
    \includegraphics[width=\textwidth]{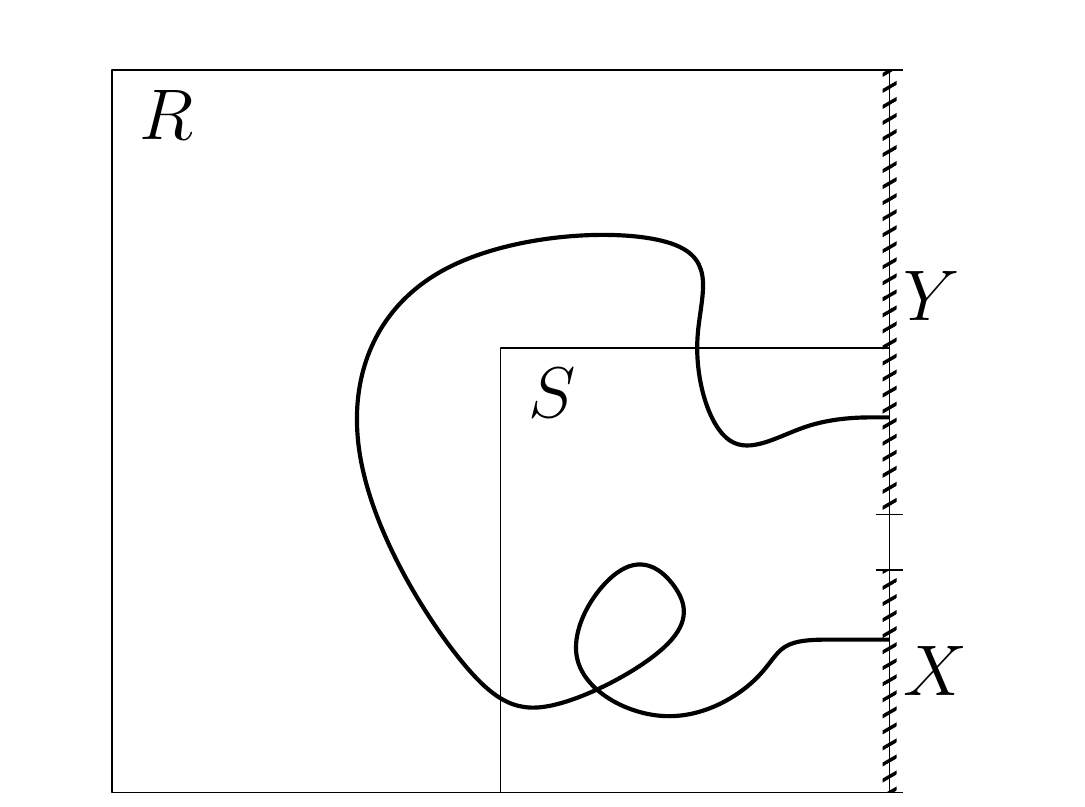}
  \end{minipage}
  \hfill
  \begin{minipage}{.43\linewidth}
    \centering
    \includegraphics[width=\textwidth]{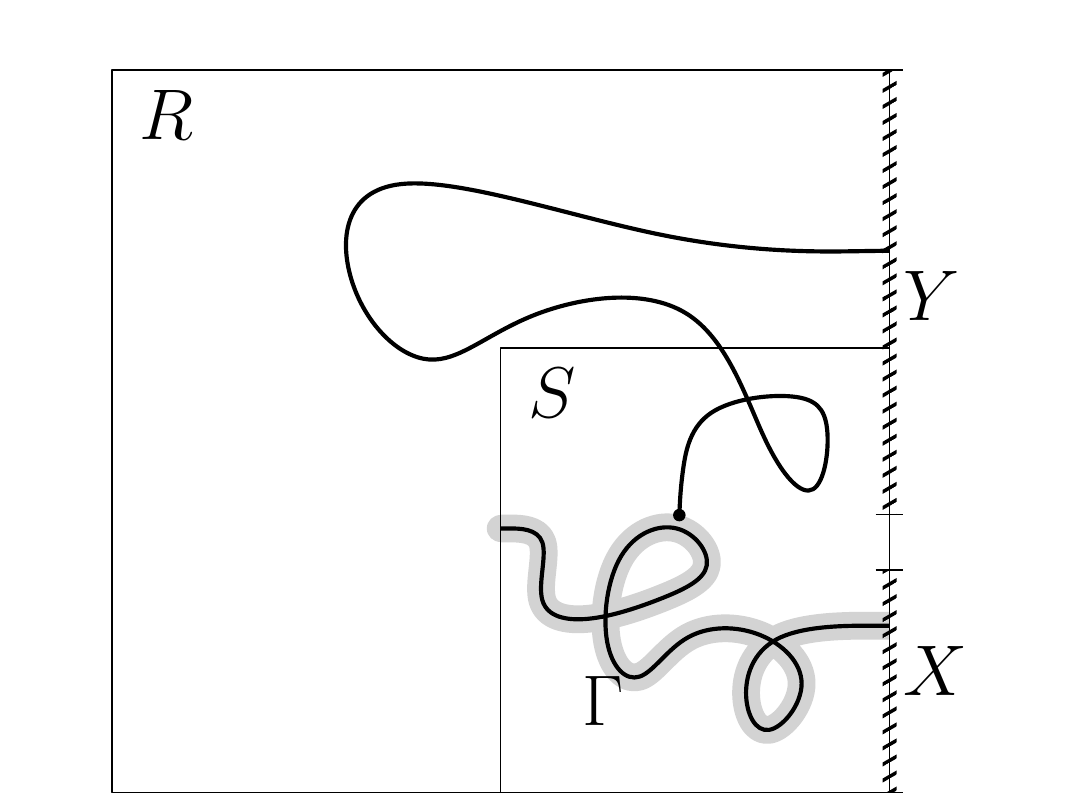}
  \end{minipage}
  \hfill
   \caption{Diagrammatic representations of the event $\mathcal
     B_1$ (on the left) and the event $\mathcal B_2$ (on the right).}
    \label{fig:3}
\end{figure}

 By the union bound, we have
    \begin{equation}
      \label{eq:1}
      c_4\le \Pp{\cal B_1}+ \Pp{\cal B_2} +\Pp{\cal A\cap \cal B_1^c \cap\cal B_2^c}.
    \end{equation}
    At least one of the three terms on the right hand side must be larger than
    $c_4/3$, and we distinguish  between these three cases. The argument for the third case will use Lemma \ref{lem:RSW3}.

 \textbf{Case 1:}  $\Pp{\cal B_1}\ge c_4/3$.

      Let $Z_i=\ol{\{7n\}\times[in,(i+1)n)}$ for $i\in \mathbb N$. Since
      $X=Z_0\cup\cdots\cup Z_3$ and $Y=Z_5\cup\cdots\cup Z_{12}$, we have

      \begin{align}
        \label{eq:51}
        c_4/3\le \Pp{\cal B_1}\le \sum_{\substack{0\le i\le3\\5\le j\le 12}}\Pp{Z_i\lr[R]Z_j}.
      \end{align}
      Therefore there exists some $i,j$ with $0\le i\le 3$ and $5\le j\le 12$ such that $\Pp{
        Z_i\lr[R]Z_j}\ge c_4/96$. We assume that
      \begin{equation}
        \label{eq:52}
        \Pp{Z_3\lr[R]Z_5}\ge c_4/96,
      \end{equation}
      since the other cases can be treated similarly. Let $R_1=[-7n,7n]\times[0,40n]$.
      By the Theorem \ref{GL0} gluing lemma and translation invariance, we have for every $6\le i \le 32$,
      \begin{align}
        \label{eq:53}
        \Pp{Z_3\lr[R_1]Z_{i}} &\ge \mathsf
        h_0(\Pp{Z_3\lr[R_1]Z_{i-1}}\wedge\Pp{Z_{i-2}\lr[R_1]Z_{i}})\\ &\ge \mathsf
        h_0(\Pp{Z_3\lr[R_1]Z_{i-1}}\wedge\Pp{Z_{3}\lr[R]Z_{5}})\\
        &\ge \mathsf
        h_0(\Pp{Z_3\lr[R_1]Z_{i-1}}\wedge (c_4/96)).
      \end{align}
      By induction, this implies that $Z_3$ is connected to $Z_{32}$ inside
      $R_1$ with probability larger than $c_5:=\mathsf h_0^{27}(c_4/96)$. When this
      holds, the rectangle $[-7n,7n]\times[4n,32n]$ is crossed from top to
      bottom by an open path. Therefore,
      \begin{equation}
        \label{eq:54}
       f(28n,14n)\ge c_5.
      \end{equation}

      %

    \textbf{Case 2:} $\Pp{\cal B_2}\ge c_4/3$.

    Apply the Theorem~\ref{GL} gluing lemma, we have
   \begin{equation}
     \label{eq:2}
     \Pp{\mathcal B_1}\ge  \mathsf h_0(\Pp{Y\lr[R] \mathcal
     N(\ol \Gamma,3r)}) =\mathsf h_0(\Pp{\cal B_2}) \ge \mathsf h_0(c_4/3).
   \end{equation}
   Then, as in the first case, there exists a constant $c_6>0$ such
   that
\begin{equation}
        \label{eq:33}
        {f(28n,14n)}\ge c_6.
      \end{equation}

\textbf{Case 3:} $\Pp{\cal A\cap \cal B_1^c \cap\cal B_2^c}\ge c_4/3$.

   When $\cal A\cap \cal B_1^c \cap\cal B_2^c$ holds the open path from
  $\mathsf B(R)$ to $Y$ must be at distance at least $3r+1$ from the minimal
  path $\Gamma=\Gamma_{\min}^S(X,\mathsf L(S))$. Define $\mathscr C$ to be the set
  of vertices that are either connected to $X$ inside $R$ or connected to a
  point $z$ whose distance from $\bar \Gamma$ satisfies $\mathrm{dist^{*}}(z,\bar
  \Gamma)\le 3r$. Alternatively, the set $\mathscr C$ can be defined by the
  following two-step exploration. First, explore all the open clusters touching $X$
  and notice that the minimal path $\Gamma$ is already determined after this first
  exploration. In a second step, explore the open clusters of all the vertices in the
  $3r$-neighborhood of $\bar\Gamma$ (that have not been explored yet). We make two
  observations:
  \begin{enumerate}[(a)]
  \item If the event $\cal A\cap \cal B_1^c \cap\cal B_2^c$ occurs, then there
    exists an open path from $Y$ to $\mathsf B(R)$ and the random set $\mathscr
    C$ does not intersect $Y$ (otherwise $\mathcal B_1$ or $\mathcal B_2$
    occurs). Therefore, the open path from $Y$ to $\mathsf B(R)$ must lie in
    $R\setminus\mathscr C$. Therefore, the event $Y \lr[R\setminus \mathscr C]
      \mathsf B(R)$ occurs.
  \item If $C$ is an admissible value for $\mathscr C$, the event $\mathscr C=C$
    is measurable with respect to the status of the edges adjacent to the set
    $C$. In particular, the status of the edges in $R\setminus  C$ is independent of
the event $\mathscr C=C$.
  \end{enumerate}

  Summing over the admissible realizations for the pair $(\mathscr C,\Gamma)$
  which allow the event $\cal A\cap \cal B_1^c \cap\cal B_2^c$ to occur, we obtain
  \begin{align}
    \label{eq:40}
    c_4/3\le\Pp{\mathcal A\cap\mathcal B_1^c\cap\mathcal B_2^c}&\overset{\hphantom{(\text{a})}}=\sum_{(C,\gamma)} \Pp{\mathcal
      A\cap\mathcal B_1^c\cap\mathcal B_2^c,\mathscr C=C, \Gamma=\gamma}\\
    &\overset{(\text{a})}\le \sum_{(C,\gamma)} \Pp{ Y \lr[R\setminus C]
      \mathsf B(R),\mathscr C=C, \Gamma=\gamma}\\
    &\overset{(\text{b})}=\sum_{(C,\gamma)} \Pp{ Y \lr[R\setminus C]
      \mathsf B(R)} \Pp{\mathscr C=C, \Gamma=\gamma} \\
    &\overset{\hphantom{(\text{a})}}\le \sum_{(C,\gamma)} \Pp{\mathcal C_\gamma}
    \Pp{\mathscr C=C, \Gamma=\gamma}\\
    &\overset{\hphantom{(\text{a})}}\le \max_\gamma\Pp{\mathcal C_\gamma}. \label{eq:5}
  \end{align}
The definition of $\mathcal C_\gamma$ and the next to last inequality here are explained in the
discussion before Lemma \ref{lem:RSW3} (see Fig.~\ref{fig:2}.)
  Equation~\eqref{eq:5} together with  Lemma~\ref{lem:RSW3} imply
  \begin{equation}
    \label{eq:6}
    f(14n,13n)\ge c_7.
  \end{equation}
  Then, by Item 1 of Proposition~\ref{prop:standard-inequality} we obtain that $f(28n,13n)\ge c_8:=\mathsf h_2^{14}(c_7)$, which implies that
  \begin{equation}
    \label{eq:49}
    f(28n,14n)\ge c_8.
  \end{equation}
Combining the three cases above, we have 
\begin{equation}
   \inf_{n \ge 1}{f(28n,14n)}\ge c_9,
  \end{equation}
with $c_9= \min\{c_5, c_6, c_8\}$, which (by another use of Proposition \ref{prop:standard-inequality}) yields the conclusion of Theorem~\ref{thm:RSW}.
\end{proof}

\subsection{RSW-Theorem: high-probability version}
\label{sec:rsw-theorem:-high}

\begin{theorem}\label{thm:RSWhp}
  Fix $\eps>0$ and $k\ge1$. For $p \in [\eps, 1-\eps]$, if $\displaystyle\sup_{n\ge1}{f(n,2n)}=1$, then
  $\displaystyle\sup_{n\ge1}{f(2n,n)}=1$.
\end{theorem}

\begin{proof}
  Assume that $\sup_{n\ge 1}f(n,2n) =1$; then by
  Proposition~\ref{prop:standard-inequality},
  \begin{equation}
    \label{eq:8}
    \sup_{n\ge1} f(3n,4n)=1.
  \end{equation}
  By Lemma~\ref{lem:FC2}, we also have that
  \begin{equation}
    \label{eq:15}
    \inf_{n\ge1}f(n,2n)>0;
  \end{equation}
  otherwise, Lemma~\ref{lem:FC2} would imply exponential decay of the one-arm event,
  which would contradict~\eqref{eq:8}. By Theorem~\ref{thm:RSW}, Eq.~\eqref{eq:15}
  implies
  \begin{equation}
    \inf_{n\ge 1}f(2n,n)>0.\label{eq:16}
\end{equation}
Using Theorem \ref{cor:GL} (and Item 1 of Proposition \ref{prop:standard-inequality}), we can fix a constant $c_0>0$ such
that, for every $n\ge1$,
\begin{equation}
  \mathbf P_p [ \text{there exists an open circuit in
    $\ol{A_{n,2n}}$ surrounding $\ol{B_n}$} ]\ge c_0.\label{eq:56}
\end{equation}
Fix $\delta>0$. By Equation~\eqref{eq:56} and independence, there exists a
constant $c_1<\infty$ large enough such that for every $n\ge 1$ and every $z\in
\mathbb Z^2$,
\begin{equation}
  \label{eq:17}
  \Pp{\text{there exists an open circuit in
    $\ol{A_{n,c_1 n}(z)}$ surrounding $\ol{B_n(z)}$}}>1-\delta,
\end{equation}

Let $R=\ol{[1,1+3c_1n]\times[-2c_1n,2c_1n]}$. By symmetry and the square root trick, there
exists $y\in \{0,n\ldots,(2c_1-1)n\}$ such that $\overline{\{1\}\times[y,y+n]}$ is
connected in $R$ to the right side $\mathsf R(R)$ with probability larger than
\begin{equation}
  \label{eq:20}
  1- (1-f(3c_1n,4c_1n))^{1/4c_1}.
\end{equation}
Therefore, by Equation~\eqref{eq:8}, we can find an $n$ such that
\begin{equation}
 \Pp{\overline{\{1\}\times[y,y+n]} \lr[R]\mathsf R(R)}\ge 1-\delta. \label{eq:14}
\end{equation}
Consider the set $S= \overline{A_{n,c_1n}(z)\setminus (\{0\}\times[y,\infty))}$
with $z=(0,y+n/2)$. Define also the sets $A=\ol{\{1\} \times [y+n/2+n, y+n/2+c_1n]}$ and $B=
\ol{\{-1\} \times [y+n/2+n, y+n/2+c_1n]}$. The key feature of these subsets (see
Fig.~\ref{fig:6}) is that the existence of an open circuit in
$\ol{A_{n,c_1n}(z)}$ with $z=(0,y+n/2)$ surrounding $\ol{B_n(z)}$ implies that
there is an open path from $A$ to $B$ inside $S$.
\begin{figure}[htbp]
  \centering
  \hfill
  \begin{minipage}{.49\linewidth}
    \includegraphics[width=\linewidth]{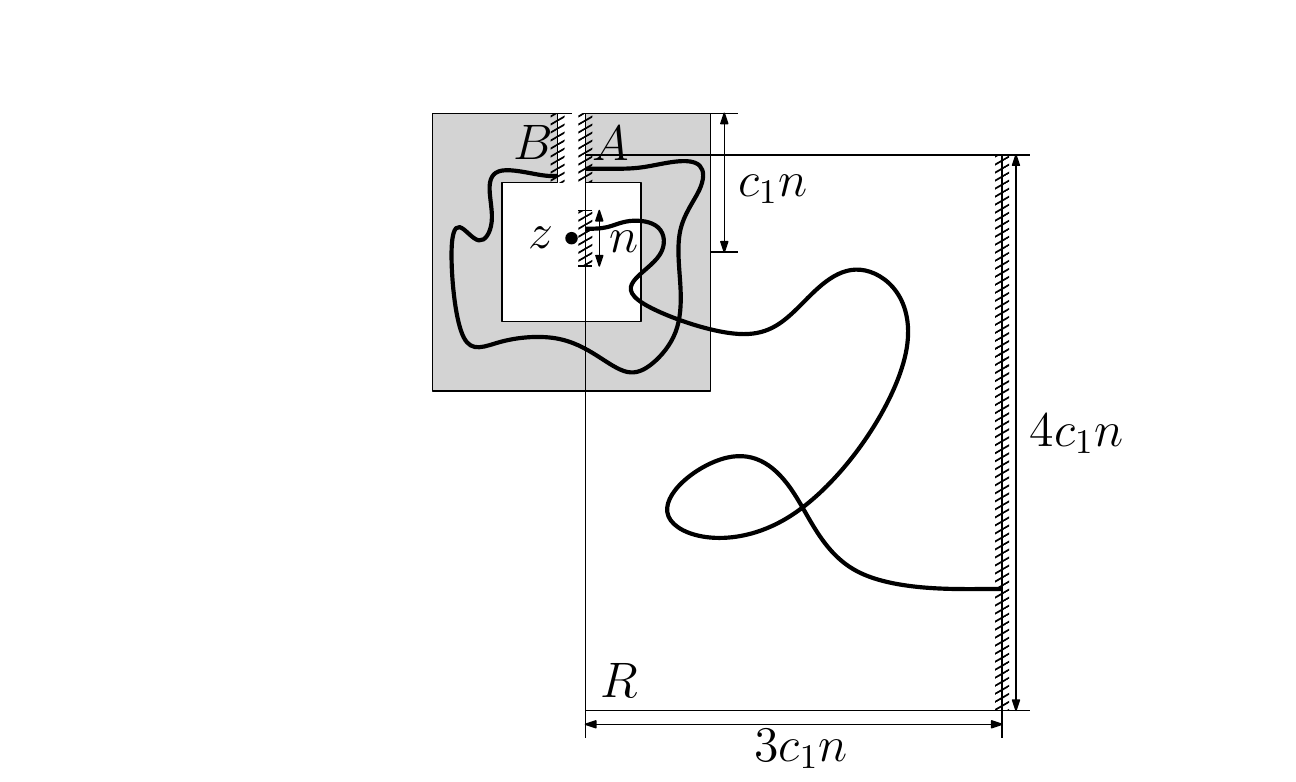}
  \end{minipage}
  \hfill
   \begin{minipage}{.49\linewidth}
    \includegraphics[width=\linewidth]{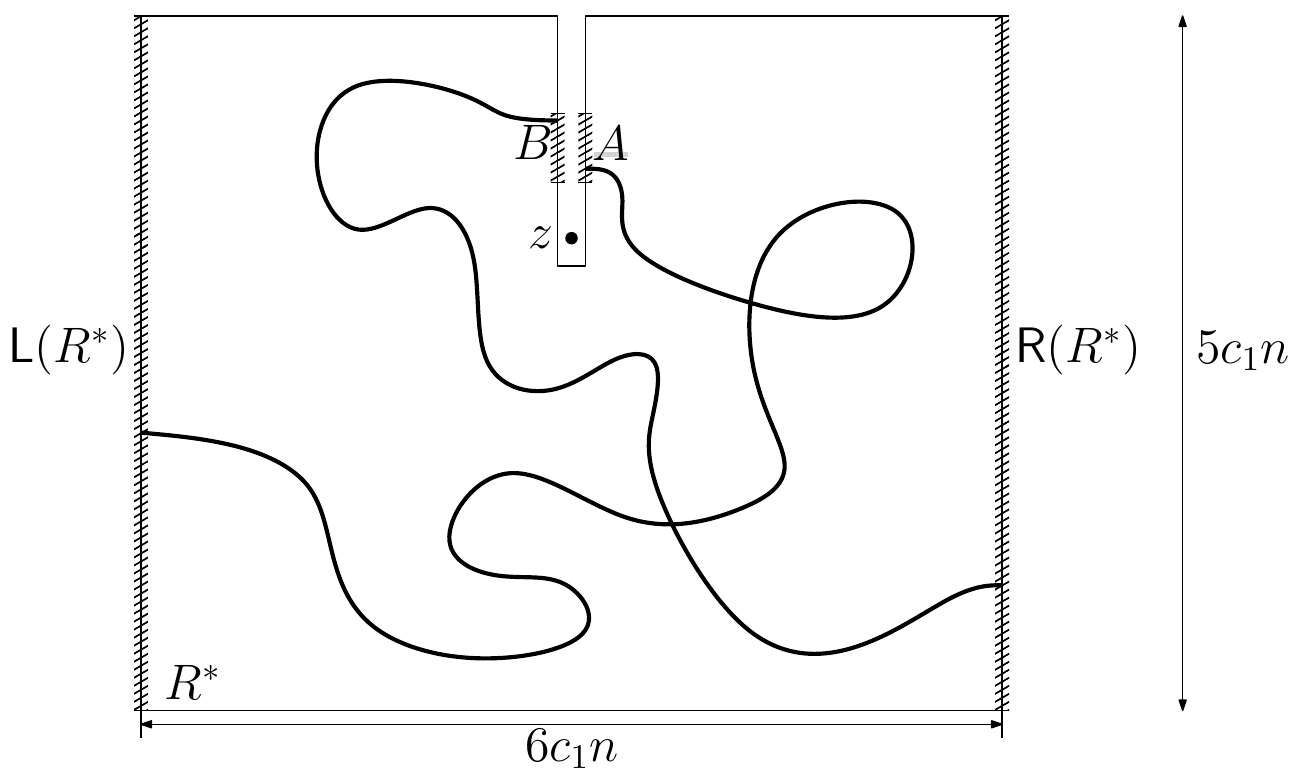}
  \end{minipage}

  \caption{Illustration of the geometric construction used to create an open
    crossing inside $R^*$. The grey region corresponds to the set $S$. First, we
    use the two open paths illustrated on the left picture to create an open
    path from $B$ to the right side of $R$. Then we use the two open paths
    illustrated on the right picture to create an open path from left to right
    in $R^*$. }
  \label{fig:6}
\end{figure}

\noindent Set $\Gamma=\Gamma_{\rm min}^S(A,B)$. Using Equations~\eqref{eq:17}
and \eqref{eq:14}, and an adaptation of the Theorem \ref{GL0} gluing lemma (where
$S$ is replaced  by $R^* = \ol{[-3c_1n,3c_1n]\times
[-2c_1n,3 c_1n] \setminus (\{0\}\times[y,\infty))}$ here), one has
\begin{equation}
  \label{eq:21}
  \Pp{B\lr[R^{*}] \mathsf R(R)}\ge \mathsf h_0\big(1-\delta\big).
\end{equation}
Let $R'$ be the symmetric reflection of $R$ through the plane $\{0\}\times \R^2$.
We have, using the Theorem \ref{GL0} gluing lemma again, that
\begin{align}
  \label{eq:22}
  f(6c_1n,5c_1n)&\ge \mathsf h_0(\Pp{B\lr[R^{*}] \mathsf R(R)} \wedge \Pp{A\lr[R^{*}] \mathsf L(R)})\\
  &\ge \mathsf h_0^{2}(1-\delta ).
\end{align}
Since $\delta>0$ was arbitrary, this completes the proof.

\end{proof}

\subsection{Proof of Theorem~\protect\ref{thm:BXP}}

\label{sec:proof-theorem1}

By Lemma~\ref{lem:renormalization-inputs} and
Theorem~\ref{thm:RSW}, we  have
\begin{equation}
  \label{eq:59}
  \inf_{n\ge1}f(2n,n)>0.
\end{equation}
By Lemma~\ref{lem:renormalization-inputs}, Item 2 of Proposition~\ref{prop:standard-inequality} and
Theorem~\ref{thm:RSWhp}, we  have
\begin{equation}
\label{eq:60}
  \sup_{n\ge1}f(n,2n)<1.
\end{equation}
Eqs.~\eqref{eq:59} and \eqref{eq:60} together with Items~\ref{item:1} and
\ref{item:2} of Proposition~\ref{prop:standard-inequality} conclude the proof.
\subsection{Proof of Corollary~\protect~\ref{cor}}
\label{sec:proof-coroll-prot}
The proofs of these items are standard. The first item follows from the RSW theorem and Theorem \ref{cor:GL}. For the
second item, note that crossing the aspect ratio $2$ annulus requires crossing
an aspect ratio $4$ rectangle, and then using the square root trick and Theorem \ref{thm:BXP}
completes the proof.
For the third item, write $\bar{B}_{n}\backslash \bar{B}_{m}$ as the
disjoint union of $\log _{2}\frac{n}{m}$ annuli, each with aspect ratio $2$, and
then the one arm probability is bounded by the probability of all
successes in $\log _{2}\frac{n}{m}$ i.i.d.\@ trials.

\section{Proof of Theorem \protect\ref{MSF}\label{sec:pf}}

Denote $C_{a,b}$ (and $D_{a,b}$) the event that there exists a $p_{c}$-open
circuit (and $p_{c}$-closed dual surface, respectively) in $\bar{B}%
_{b}\backslash \bar{B}_{a}$ that surrounds the origin. By using Corollary \ref{cor}%
, it is easy to see that we can choose two alternating sequences $\left\{ n_{i}\right\} ,\left\{
m_{i}\right\} $, such that for each $i$, $2n_{i}<m_{i}<n_{i+1}$, and
\begin{eqnarray}
\mathbf{P}_{p_{c}}\left[ C_{n_{i},2n_{i}}\right]  &\geq &c_{0},   \\
\mathbf{P}_{p_{c}}\left[ D_{2n_{i},m_{i}}\right]  &\geq &1-c_{0}/2.
\end{eqnarray}

We use the total order on circuits defined before the proof of Theorem \ref{cor:GL2}. Given $\omega \in C_{n_{i},2n_{i}}$, we define $\Gamma _{\min }^{\left(
i\right) }\left( \omega \right) $ to be the minimal $p_{c}$-open circuit in $%
\bar{B}_{2n_{i}}\backslash \bar{B}_{n_{i}}$ that surrounds the origin. We
will omit the superscript $i$ when it is clear from the context.

For $x\in \mathbb{S}_{k}$ and $n\in \mathbb{N}$ (with $x$ in $\bar{B}_{n}$), we denote by $\mathcal{I}_{x}^{n}$
the invasion cluster starting at $x$, and stopped when it first reaches
any vertex in $\partial \bar{B}_{n}$. Let $\mathcal{B}_{x}^{m_{i}}=\left\{
\omega :\Gamma _{\min }^{\left( i\right) }\left( \omega \right) \subset
\mathcal{I}_{x}^{m_{i}}\left( \omega \right) \right\} $. For any Borel
measurable set $A\subset \left[ 0,1\right] ^{\bar{B}_{m_{i-1}}}$, denote $%
\mathcal{Y}_{A}^{i}=C_{n_{i},2n_{i}}\cap D_{2n_{i},m_{i}}\cap A$. The
following lemma will be proved in Section \ref{gluepf}.

\begin{lemma}[Gluing lemma for invasion]
\label{glue}Fix $x\in \mathbb{S}_{k}$. Take $i_{0}$ such that $x\in \bar{B}%
_{m_{i_{0}}}$. Then for any $i>i_{0}$, and any Borel measurable set $%
A\subset \left[ 0,1\right] ^{\bar{B}_{m_{i-1}}}$, there exist $%
C_{1},C_{2},C_{3}<\infty $, such that%
\begin{align}
\mathbb{P}\left[ \left( \mathcal{B}_{0}^{m_{i}}\right) ^{c},\mathcal{B}%
_{x}^{m_{i}},\mathcal{Y}_{A}^{i}\right]  &\leq C_{1}\mathbb{P}\left[
\mathcal{B}_{0}^{m_{i}},\mathcal{B}_{x}^{m_{i}},\mathcal{Y}_{A}^{i}\right] ,
\label{g1} \\
\mathbb{P}\left[ \mathcal{B}_{0}^{m_{i}},\left( \mathcal{B}%
_{x}^{m_{i}}\right) ^{c},\mathcal{Y}_{A}^{i}\right]  &\leq C_{2}\mathbb{P}%
\left[ \mathcal{B}_{0}^{m_{i}},\mathcal{B}_{x}^{m_{i}},\mathcal{Y}_{A}^{i}%
\right] ,  \label{g2} \\
\mathbb{P}\left[ \left( \mathcal{B}_{0}^{m_{i}}\right) ^{c},\left( \mathcal{B%
}_{x}^{m_{i}}\right) ^{c},\mathcal{Y}_{A}^{i}\right]  &\leq C_{3}\mathbb{P}%
\left[ \mathcal{B}_{0}^{m_{i}},\mathcal{B}_{x}^{m_{i}},\mathcal{Y}_{A}^{i}%
\right] .  \label{g3}
\end{align}%
As a consequence, there exist $c_{1},c_{2}>0$, such that
\begin{equation*}
\mathbb{P}\left[ \mathcal{B}_{0}^{m_{i}},\mathcal{B}_{x}^{m_{i}},\mathcal{Y}%
_{A}^{i}\right] \geq c_{2}\mathbb{P}\left[ \mathcal{Y}_{A}^{i}\right] \geq
c_{1}\mathbb{P}\left[ A\right] .
\end{equation*}
\end{lemma}

Assuming the lemma, we now complete the proof of Theorem \ref{MSF}. Denote $%
\mathcal{Z}^{i}=\mathcal{B}_{0}^{m_{i}}\cap \mathcal{B}_{x}^{m_{i}}\cap
C_{n_{i},2n_{i}}$. Since $\mathcal{Z}^{i-1}$ is measurable with respect to
the state of edges in $\bar{B}_{m_{i-1}}$, for all $i$ sufficiently large,
\begin{equation*}
\mathbb{P}\left[ \mathcal{Z}^{i}|\left( \mathcal{Z}^{i_{0}}\right)
^{c},\left( \mathcal{Z}^{i_{0}+1}\right) ^{c},...,\left( \mathcal{Z}%
^{i-1}\right) ^{c}\right] \geq c_{1}.
\end{equation*}%
It then follows by comparison to a sequence of i.i.d trials with success probability $c_{1}$ that
\begin{equation*}
\mathbb{P}\left( \mathcal{B}_{0}^{m_{i}},\mathcal{B}%
_{x}^{m_{i}},C_{n_{i},2n_{i}}\text{ i.o.}\right) =1.
\end{equation*}%
Finally, notice that since $\mathcal{B}_{0}^{m_{i}}\subset \left\{ \omega :\Gamma
_{\min }^{\left( i\right) }\left( \omega \right) \subset \mathcal{I}%
_{0}\left( \omega \right) \right\} $, $\mathcal{B}_{x}^{m_{i}}\subset
\left\{ \omega :\Gamma _{\min }^{\left( i\right) }\left( \omega \right)
\subset \mathcal{I}_{x}\left( \omega \right) \right\} $, we can conclude that%
\begin{equation*}
\mathbb{P}\left[ \Gamma _{\min }^{\left( i\right) }\subset \mathcal{I}%
_{0},\Gamma _{\min }^{\left( i\right) } \subset
\mathcal{I}_{x},C_{n_{i},2n_{i}}\text{ i.o.}\right] =1.
\end{equation*}%
In particular, $\mathcal{I}_{0}\cap \mathcal{I}_{x}\neq \emptyset $ a.s.

\subsection{Proof of Lemma \protect\ref{glue}\label{gluepf}}
In order to prove Lemma \ref{glue}, we start with the following extension of the combinatorial
Lemma 7 of \cite{DST}. It concerns maps $\Phi$ on the edge labels $\omega = \{\omega(e), e \in E \}$ such that
$\Phi$ decreases (respectively, increases) finitely many $\omega(e)$'s in an affine way in order to make those
edges open (respectively, closed).

\begin{lemma}
\label{combi}Consider $\mathcal{A},\mathcal{B}\subset \mathcal{F}$, $a,b \in (0,1)$, and a measurable map
$\Phi :\mathcal{A}\rightarrow \mathcal{B}$. If for any $\omega ^{\prime }\in
\Phi \left( \mathcal{A}\right) $, there exists $S\left( \omega ^{\prime
}\right) \subset E$ with less than or equal to $s$ edges, such that%
\begin{eqnarray*}
\Phi ^{-1}\left( \omega ^{\prime }\right)  &\subset &\left\{ \omega :\omega
|_{S^{c}}=\omega ^{\prime }|_{S^{c}}\right\} \cap  \\
&&\left[ \cup _{L\subset S}\left( \left\{ \omega :\omega |_{L}=\frac{1}{a}%
\omega ^{\prime }|_{L}\right\} \cap \left\{ \omega :\omega |_{S\backslash L}=%
\frac{\omega ^{\prime }-b}{1-b}|_{S\backslash L}\right\} \right) \right] ,
\end{eqnarray*}%
then $\mathbb{P}\left[ \mathcal{A}\right] \leq \left( \frac{2}{a\wedge
\left( 1-b\right) }\right) ^{s}\mathbb{P}\left[ \mathcal{B}\right] .$
\end{lemma}

Roughly speaking, this lemma says that if one can obtain $\mathcal{B}$ by modifying a small
number of edges in $\mathcal{A}$, in a way that given any element in $%
\mathcal{B}$, the number of its pre-images is bounded, then $\mathbb{P}\left[
\mathcal{A}\right] $ can be bounded from above by a constant times $\mathbb{P}\left[ \mathcal{%
B}\right] $.

\noindent\textit{Remarks.}
\begin{enumerate}[1.]
\item An equivalent way of stating the hypotheses on $\Phi$ is that $\Phi$ leaves all but at most $s$ of the
$\omega(e)$'s unchanged, with the others either lowered (by $\omega(e) \mapsto a\omega(e)$), or raised
(by  $\omega(e) \mapsto b+ (1-b)\omega(e)$) and the set $S$ of changed edges is uniquely determined by
$\omega^{\prime} = \Phi(\omega)$.
\item In Lemma \ref{combi}, all the edges in $S$ have their edge labels either decreased or increased. Although
it is not needed in this paper, we note that the lemma can be extended to allow for some of the edges in $S$ to be
unchanged.
\end{enumerate}

\begin{proof}[Proof of Lemma \ref{combi}]
First observe that for any $\omega ^{\prime }\in \Phi \left( \mathcal{A}\right) $%
, Card$\left( \Phi ^{-1}\left( \omega ^{\prime }\right) \right) \leq
2^{\left\vert S\left( \omega ^{\prime }\right) \right\vert }\leq 2^{s}$.
Therefore one can take a disjoint partition $\left\{ A_{i}\right\}
_{i=1}^{2^{s}}$ (some of which may be empty) of $\mathcal{A}$, such that $%
\Phi |_{A_{i}}$ is a bijection. Indeed, there is an $L_{i}(\omega)$ for $\omega \in A_{i}$, such that %
\begin{equation*}
\Phi |_{A_{i}}\left( \omega \right) \left( e\right) =\left\{
\begin{array}{cc}
\omega \left( e\right)  & \text{if }e\notin S\left( \Phi \left( \omega
\right) \right)  \\
b+\left( 1-b\right) \omega \left( e\right)  & e\in S\left( \Phi \left(
\omega \right) \right) \setminus L_{i} \\
a\omega \left( e\right)  & e\in L_{i}%
\end{array}%
\right. .
\end{equation*}%
Then one can bound its Jacobian $J_{i}(\omega)$ from below by%
\begin{equation*}
J_{i}(\omega) \geq a^{\text{Card}\left( L_{i}\right)
}\left( 1-b\right) ^{\text{Card}(S\left( \Phi \left( \omega \right) \right)
\setminus L_{i})}\geq \left( a\wedge \left( 1-b\right) \right) ^{s}.
\end{equation*}%
Therefore,%
\begin{eqnarray*}
\mathbb{P}[\mathcal{B]} &\geq &\int_{\Phi \left( A_{i}\right) }d\omega
^{\prime }=\int_{A_{i}}J_{i}(\omega)  d\omega  \\
&\geq &\int_{A_{i}}\left( a\wedge \left( 1-b\right) \right) ^{s}d\omega
=\left( a\wedge \left( 1-b\right) \right) ^{s}\mathbb{P[}A_{i}].
\end{eqnarray*}%
Summing over $i$, we obtain
\begin{equation*}
2^{s}\mathbb{P}[\mathcal{B]}\geq \left( a\wedge \left( 1-b\right) \right)
^{s}\mathbb{P}[\mathcal{A]}.
\end{equation*}
\end{proof}

\begin{proof}[Proof of Lemma \ref{glue}]
We now prove \eqref{g1} by explicitly constructing a map $\Phi :\left(
\mathcal{B}_{0}^{m_{i}}\right) ^{c}\cap \mathcal{B}_{x}^{m_{i}}\cap \mathcal{%
Y}_{A}^{i}\rightarrow \mathcal{B}_{0}^{m_{i}}\cap \mathcal{B}%
_{x}^{m_{i}}\cap \mathcal{Y}_{A}^{i}$ that satisfies the hypothesis of Lemma %
\ref{combi}. The proof of \eqref{g2} follows from the same argument, and the proof of \eqref{g3} will be described at the end of this proof. Given $%
\omega \in \left( \mathcal{B}_{0}^{m_{i}}\right) ^{c}\cap \mathcal{B}%
_{x}^{m_{i}}\cap C_{n_{i},2n_{i}}\cap D_{2n_{i},m_{i}}$ with $i>i_{0}$, let $%
R\left( \omega \right) $ denote the connected component containing $0$ in $%
\left\{ w\in \mathbb{S}_{k}\text{: dist}\left( \bar{w},\Gamma _{\min }\left(
\omega \right) \right) \ge 1\right\} $. In particular, $\partial R \subset \left\{ w\in \mathbb{S}_{k}\text{: dist}\left( \bar{w},\Gamma _{\min }\left(
\omega \right) \right) =1\right\}$. Let $\tau _{i}=\min \left\{ j:\mathcal{I%
}_{0}\left[ j\right] \in \partial R\right\} $ and $z(\omega)=\mathcal{I}_{0}\left[
\tau _{i}\right] $ be the first landing point of $\mathcal{I}_{0}$ on $%
\partial R$. By definition, there exists $z^{\prime }\in \Gamma _{\min }$
such that dist$\left( \bar{z},\bar{z^{\prime }}\right) =1$. If there exists
more than one $z^{\prime }$ satisfy dist$\left( \bar{z},\bar{z^{\prime }}\right)
=1$, choose the minimal one.

Recall that $\mathcal{C}_{p_{c}}\left( \Gamma _{\min }\left( \omega \right)
\right) $ denotes the $p_{c}$-open cluster containing $\Gamma _{\min }\left(
\omega \right) $. Notice that $\omega \in \left( \mathcal{B}%
_{0}^{m_{i}}\right) ^{c}$ implies

\begin{equation}
z\left( \omega \right) \notin \mathcal{C}_{p_{c}}\left( \Gamma _{\min
}\left( \omega \right) \right) .  \label{f1}
\end{equation}%
Otherwise, by the observation in Section \ref{inv}, we would have $\Gamma _{\min }\left(
\omega \right) \subset \mathcal{I}_{0}^{m_{i}}\left( \omega \right) $.
To complete the proof, we will use $B_{1}^{\#}(z)$ for $ z\in \mathbb{S}_{k}$ to denote\\
$ z+\{(0,0,0),(1,0,0),(-1,0,0),(0,1,0),(0,-1,0)\}$; as usual, $\bar{B}_{1}^{\#}\left( z\right)$
denotes the cylinder in $\mathbb{S}_k$ generated by the five-point set $B_{1}^{\#}(z)$.
There exists a self avoiding path $\Gamma _{z}$ in $\bar{B}%
_{1}^{\#}\left( z^{\prime }\right) $ connecting $z\left( \omega \right) $ to $%
\Gamma _{\min }\left( \omega \right)$ without touching any other vertices
in $\Gamma _{\min }\left( \omega \right)
\cup \partial \bar{B}_{1}^{\#}\left( z^{\prime }\right) $. In particular one can construct $\Gamma _{z}$ by taking
one edge from $z$ to $\bar{z^{\prime }}$ and then move in $\bar{z^{\prime }}$ until reaching the first vertex in $\Gamma _{\min }$.

We now construct $\omega ^{\prime }=\Phi \left( \omega \right) $ as follows.

\begin{enumerate}

\item Open all
the edges in $\Gamma _{z}$ (that is, take $\omega \left( e\right) \mapsto
p_{c}\omega \left( e\right) $).

\item If $\mathcal{I}_{x}^{m_{i}}$ touches $\bar{B}_{1}^{\#}\left( z^{\prime }\right) $,
then proceed as follows; otherwise go to Step 3. Define $\tau _{i}^{x}=\min \left\{ j:\mathcal{I}_{x}\left[
j\right] \in \partial \bar{B}_{1}^{\#}\left( z^{\prime }\right) \right\} $ and $w=%
\mathcal{I}_{x}\left[ \tau _{i}\right] $ to be the first vertex in $\partial
\bar{B}_{1}^{\#}\left( z^{\prime }\right) $ reached by the invasion cluster
starting from $x$. If $w\in \mathcal{C}_{p_{c}}\left( z\right) $, go to Step 3. If $w\notin \mathcal{C}_{p_{c}}\left( z\right) $, then there
exists a self avoiding path $\Gamma _{w}$ in $\bar{B}_{1}^{\#}\left( z^{\prime
}\right) $ connecting $w\left( \omega \right) $ to $ \Gamma _{\min }\left( \omega \right) \cup \Gamma _{z}$ without touching any other
vertices in $ \Gamma _{\min }\left( \omega \right)
 \cup \Gamma _{z}\cup \partial \bar{B}_{1}^{\#}\left( z^{\prime }\right) $%
. Open all the edges in $\Gamma _{w}$.

\item Close all the edges in $\bar{B}_{1}^{\#}\left( z^{\prime }\right) $ (that
is, map $\omega \left( e\right) \mapsto b+\left( 1-b\right) \omega \left(
e\right) $ with $b>p_{c}$) except for the edges of $\Gamma _{\min }\left( \omega \right) \cup \Gamma _{z} \cup \Gamma _{w} $.
\end{enumerate}

By construction, $\omega ^{\prime }\in \mathcal{B}_{0}^{m_{i}}\cap \mathcal{B%
}_{x}^{m_{i}}\cap \mathcal{Y}_{A}^{i}$. To see this, note that when $%
D_{2n_{i},m_{i}}$ occurs, $\mathcal{I}_{0}^{m_{i}}$ (or $\mathcal{I}%
_{x}^{m_{i}}$) touching any vertex $v\in \bar{B}_{2n_{i}}\left( 0\right) $
implies it also contains all of $\mathcal{C}_{p_{c}}\left( v\right) $. To apply Lemma \ref%
{combi}, we need to bound the number of pre-images of $\omega ^{\prime }$ by $2^{s}$ for some $s$.
For this, we first note the important feature that
\begin{equation}
\Gamma _{\min }\left( \omega ^{\prime }\right) =\Gamma _{\min }\left( \omega
\right) .  \label{f2}
\end{equation}
This will help show that the set $S$ of changed edges is uniquely determined by $\omega ^{\prime }$.

Indeed, the construction will not create any new $p_{c}$-open circuits. If
the construction skips Step 2, then any new $p_{c}$-open circuit would contain a
subset of $\Gamma _{z}$, and then it would contain all of $\Gamma _{z}$ because of
Step 3. Therefore if the construction created some new $p_{c}$-open circuit,
we would have $z\left( \omega \right) \in \mathcal{C}_{p_{c}}\left( \Gamma
_{\min }\left( \omega \right) \right) $, which would contradict \eqref{f1}. If
the construction uses Step 2, by the same argument, we would have either $%
z\left( \omega \right) \in \mathcal{C}_{p_{c}}\left( \Gamma _{\min }\left(
\omega \right) \right) $, or $w\left( \omega \right) \in \mathcal{C}%
_{p_{c}}\left( \Gamma _{\min }\left( \omega \right) \right) $, or $w\left(
\omega \right) \in \mathcal{C}_{p_{c}}\left( z\left( \omega \right) \right) $%
, any of which would lead to a contradiction.

Now, given $\omega ^{\prime } = \Phi(\omega)$, one can determine $S(\omega)= S(\omega ^{\prime }) $ as follows.

\begin{enumerate}
\item Thanks to \eqref{f2}, $R(\omega ^{\prime })=R\left( \omega \right) $,
and $\omega |_{R\left( \omega \right) }=\omega ^{\prime }|_{R\left( \omega
^{\prime }\right) }$ (where $\omega |_{R}$ here means the set of edge labels with both vertices in $R$). This implies $z\left( \omega \right) =z\left( \omega%
^{\prime }\right) $. Therefore one can explore $\mathcal{I}%
_{0}^{m_{i}}\left( \omega ^{\prime }\right) $ until it contains $z$,
without any change from $\mathcal{I}%
_{0}^{m_{i}}\left( \omega \right) $.

\item $z^{\prime }\left( \omega ^{\prime }\right) =z^{\prime }\left( \omega
\right) $. In fact, $z^{\prime }\in \Gamma _{\min }\left( \omega ^{\prime
}\right) =\Gamma _{\min }\left( \omega \right) $ is uniquely characterized
by dist$\left( \bar{z},\bar{z^{\prime }}\right) =1$ and the minimality of
 $z^{\prime }$.

\item Taking $S\left( \omega ^{\prime }\right) =\bar{B}_{1}^{\#}\left( z^{\prime
}\right) \setminus \Gamma_{\min}(\omega ^{\prime })$, we see that $\omega |_{S^{c}}=\omega ^{\prime }|_{S^{c}}$, and
that the map $\Phi $ satisfies the conditions of Lemma \ref{combi} with $s$ equal
to the number of edges in $\bar{B}_{1}^{\#}$. Applying Lemma \ref{combi} we
obtain that
\begin{eqnarray*}
&&\mathbb{P}\left[ \left( \mathcal{B}_{0}^{m_{i}}\right) ^{c},\mathcal{B}%
_{x}^{m_{i}},\mathcal{Y}_{A}^{i}\right]  \\
&\leq &\left( \frac{2}{p_{c}\wedge \left( 1-b\right) }\right) ^{s}\mathbb{P}%
\left[ \mathcal{B}_{0}^{m_{i}},\mathcal{B}_{x}^{m_{i}},\mathcal{Y}_{A}^{i}%
\right] ,
\end{eqnarray*}%
which concludes the proof of \eqref{g1} (and similarly \eqref{g2}).
\end{enumerate}

Finally, to prove \eqref{g3}, we note that a map
\begin{equation*}
\Phi :\left( \mathcal{B}_{0}^{m_{i}}\right) ^{c}\cap \left( \mathcal{B}%
_{x}^{m_{i}}\right) ^{c}\cap \mathcal{Y}_{A}^{i}\rightarrow \left( \mathcal{B%
}_{0}^{m_{i}}\cap \mathcal{B}_{x}^{m_{i}}\cap \mathcal{Y}_{A}^{i}\right)
\cup \left( \mathcal{B}_{0}^{m_{i}}\cap \left( \mathcal{B}%
_{x}^{m_{i}}\right) ^{c}\cap \mathcal{Y}_{A}^{i}\right)
\end{equation*}%
can be constructed in essentially the same way as above (in fact, one can skip Step 2 when constructing $\omega ^{\prime }$). Lemma \ref{combi} then
implies%
\begin{equation*}
\mathbb{P}\left[ \left( \mathcal{B}_{0}^{m_{i}}\right) ^{c},\left( \mathcal{B%
}_{x}^{m_{i}}\right) ^{c},\mathcal{Y}_{A}^{i}\right] \leq C_{4}\mathbb{P}%
\left[ \mathcal{B}_{0}^{m_{i}},\mathcal{B}_{x}^{m_{i}},\mathcal{Y}_{A}^{i}%
\right] +C_{4}\mathbb{P}\left[ \mathcal{B}_{0}^{m_{i}},\left( \mathcal{B}%
_{x}^{m_{i}}\right) ^{c},\mathcal{Y}_{A}^{i}\right] ,
\end{equation*}%
for some $C_{4}<\infty $. Together with \eqref{g2} we obtain \eqref{g3} with $C_{3}= C_{4}(1+C_{2})$.
\end{proof}

\bigskip
\paragraph{\textbf{Acknowledgments:}} We thank Artem Sapozhnikov for very helpful
comments on earlier versions of the paper and Vladas Sidoravicius for useful
comments. The research of C.M.N. and W.W. was supported in part by U.S. NSF
grants DMS-1007524 and DMS-1507019. The research of V.T. was supported by the
Swiss NSF.

\bibliographystyle{alpha}
\bibliography{MSFref}

\begin{thebibliography}{DST15b}

\bibitem[AKN87]{AKN}
Michael Aizenman, Harry Kesten, and Charles~M Newman.
\newblock Uniqueness of the infinite cluster and continuity of connectivity
  functions for short and long range percolation.
\newblock {\em Communications in Mathematical Physics}, 111(4):505--531, 1987.

\bibitem[Ale95]{Alex}
Kenneth~S Alexander.
\newblock Percolation and minimal spanning forests in infinite graphs.
\newblock {\em The Annals of Probability}, pages 87--104, 1995.

\bibitem[AM94]{AM}
Kenneth~S Alexander and Stanislav~A Molchanov.
\newblock Percolation of level sets for two-dimensional random fields with
  lattice symmetry.
\newblock {\em Journal of Statistical Physics}, 77(3-4):627--643, 1994.

\bibitem[ATT]{alhberg2015continuum}
D.~Ahlberg, V.~Tassion, and A.~Teixeira.
\newblock Sharpness of the phase transition for continuum percolation in the
  plane.
\newblock In preparation.

\bibitem[BBW05]{balister2005continuum}
P.~Balister, B.~Bollob{\'a}s, and M.~Walters.
\newblock Continuum percolation with steps in the square or the disc.
\newblock {\em Random Structures Algorithms}, 26(4):392--403, 2005.

\bibitem[BK89]{BK}
Robert~M Burton and Michael Keane.
\newblock Density and uniqueness in percolation.
\newblock {\em Communications in mathematical physics}, 121(3):501--505, 1989.

\bibitem[BLPS01]{BLPS}
Itai Benjamini, Russell Lyons, Yuval Peres, and Oded Schramm.
\newblock Special invited paper: uniform spanning forests.
\newblock {\em Annals of probability}, pages 1--65, 2001.

\bibitem[BS15]{sapozhnikov2015crossing}
D.~Basu and A.~{Sapozhnikov}.
\newblock {C}rossing probabilities for critical {B}ernoulli percolation on
  slabs.
\newblock arXiv:1512.05178, 2015.

\bibitem[CCN85]{CCN}
Jennifer~T Chayes, Lincoln Chayes, and Charles~M Newman.
\newblock The stochastic geometry of invasion percolation.
\newblock {\em Communications in Mathematical Physics}, 101(3):383--407, 1985.

\bibitem[DST15a]{DST}
H.~{Duminil-Copin}, V.~Sidoravicius, and V.~Tassion.
\newblock Absence of infinite cluster for critical {B}ernoulli percolation on
  slabs.
\newblock {\em Communications in Pure and Applied Mathematics}, to appear,
  2015.

\bibitem[DST15b]{duminil2015continuity}
H.~{Duminil-Copin}, V.~Sidoravicius, and V.~Tassion.
\newblock Continuity of the phase transition for planar random-cluster and
  {P}otts models with $1\le q\le 4$.
\newblock arXiv:1505.04159, 2015.

\bibitem[GPS14]{GPS}
Christophe Garban, G{\'a}bor Pete, and Oded Schramm.
\newblock The scaling limits of the minimal spanning tree and invasion
  percolation in the plane.
\newblock {\em Annals of Probability}, 2014.

\bibitem[Gri99]{grimmett1999}
G.~R. Grimmett.
\newblock {\em Percolation}, volume 321.
\newblock Springer Verlag, 1999.

\bibitem[H{\"a}g95]{Hagg}
Olle H{\"a}ggstr{\"o}m.
\newblock Random-cluster measures and uniform spanning trees.
\newblock {\em Stochastic processes and their applications}, 59(2):267--275,
  1995.

\bibitem[Jac10]{Jack}
Thomas~Sundal Jackson.
\newblock {\em Properties of minimum spanning trees and fractional quantum Hall
  states}.
\newblock 2010.

\bibitem[Kes82]{kesten1982percolation}
H.~Kesten.
\newblock {\em Percolation theory for mathematicians}, volume~2 of {\em
  Progress in Probability and Statistics}.
\newblock Birkh\"auser Boston, Mass., 1982.

\bibitem[Kes87]{kesten1987scaling}
H.~Kesten.
\newblock Scaling relations for 2{D}-percolation.
\newblock {\em Communications in Mathematical Physics}, 109(1):109--156, 1987.

\bibitem[LPS06]{LPS}
Russell Lyons, Yuval Peres, and Oded Schramm.
\newblock Minimal spanning forests.
\newblock {\em The Annals of Probability}, pages 1665--1692, 2006.

\bibitem[LSS97]{ligget1997domination}
T.~M. Liggett, R.~H. Schonmann, and A.~M. Stacey.
\newblock Domination by product measures.
\newblock {\em Ann. Probab.}, 25(1):71--95, 1997.

\bibitem[NS94]{NS94}
CM~Newman and DL~Stein.
\newblock Spin-glass model with dimension-dependent ground state multiplicity.
\newblock {\em Physical review letters}, 72(14):2286, 1994.

\bibitem[NS96]{NS}
CM~Newman and DL~Stein.
\newblock Ground-state structure in a highly disordered spin-glass model.
\newblock {\em Journal of statistical physics}, 82(3-4):1113--1132, 1996.

\bibitem[Pem91]{Pem}
Robin Pemantle.
\newblock Choosing a spanning tree for the integer lattice uniformly.
\newblock {\em The Annals of Probability}, pages 1559--1574, 1991.

\bibitem[Rus78]{russo1978note}
L.~Russo.
\newblock A note on percolation.
\newblock {\em Zeitschrift f{\"u}r Wahrscheinlichkeitstheorie und verwandte
  Gebiete}, 43(1):39--48, 1978.

\bibitem[Smi01]{smirnov2001critical}
S.~Smirnov.
\newblock Critical percolation in the plane: conformal invariance, {C}ardy's
  formula, scaling limits.
\newblock {\em Comptes Rendus de l'Acad{\'e}mie des Sciences. S{\'e}rie I.
  Math{\'e}matique}, 333(3):239--244, 2001.

\bibitem[SW78]{seymour1978percolation}
P.~D. Seymour and D.~J.~A. Welsh.
\newblock Percolation probabilities on the square lattice.
\newblock {\em Ann. Discrete Math.}, 3:227--245, 1978.
\newblock Advances in graph theory (Cambridge Combinatorial Conf., Trinity
  College, Cambridge, 1977).

\bibitem[Tas15]{tassion2015crossing}
V.~Tassion.
\newblock Crossing probabilities for {V}oronoi percolation.
\newblock {\em Annals of Probability}, To appear, 2015.
\newblock arXiv:1410.6773.

\bibitem[Wil96]{Wil}
David~Bruce Wilson.
\newblock Generating random spanning trees more quickly than the cover time.
\newblock In {\em Proceedings of the twenty-eighth annual ACM symposium on
  Theory of computing}, pages 296--303. ACM, 1996.

\end{thebibliography}

\end{document}